%% file: optimal-HDG-arxiv.tex
\documentclass{amsart}
\usepackage{url}
\usepackage{amsmath, amsthm, amssymb, amsfonts, verbatim}
\usepackage{stmaryrd}

\usepackage[active]{srcltx}
\usepackage{array}
\usepackage[bookmarks]{hyperref}
\usepackage{amscd}
\usepackage{epsfig}
\usepackage{color}
\usepackage{comment}
\usepackage{setspace}
\usepackage{textcomp}
\usepackage{url}
\usepackage{multirow}
\usepackage{stmaryrd}
\usepackage{pdfsync}
\usepackage{breqn}

\numberwithin{equation}{section}

\input{jlmacros}

\usepackage{xspace,color}

\newcommand{\betab}{\bs{\beta}}

\newcommand{\dif}{\,{\rm d}}

\newcommand{\kappab}{\boldsymbol{\kappa}}

\newcommand{\nb}{\boldsymbol{n}}
\newcommand{\norm}[1]{\|  #1 \|} 
\newcommand{\normw}[2]{\left\Vert#1\right\Vert_{#2}}

\newcommand{\pbar}{\bar{p}}

\newcommand{\qbar}{\bar{q}}

\newcommand{\sigmab}{\boldsymbol{\sigma}}
\newcommand{\taub}{\boldsymbol{\tau}}

\newcommand{\ub}{\boldsymbol{u}}
\newcommand{\ubbar}{\bs{\bar{u}}}
\newcommand{\ubar}{\bar{p}}

\newcommand{\vb}{\boldsymbol{v}}
\newcommand{\vbbar}{\boldsymbol{\bar{v}}}
\newcommand{\Vb}{\boldsymbol{V}}

\newcommand{\wb}{\boldsymbol{w}}

\newcommand{\Xb}{\boldsymbol{X}}

\newcommand{\Yb}{\boldsymbol{Y}}

\newtheorem{remark}[theorem]{Remark}

\begin{document}


\title[HDG with elliptic regularity]{Analysis of hybridized discontinuous Galerkin methods without elliptic regularity assumptions}

\author{Jeonghun J. Lee} 
\address{Department of Mathematics, Baylor University, Waco, TX , USA}
\email{jeonghun\_lee@baylor.edu}
\urladdr{}
\date{November, 2019 }
\maketitle

\begin{abstract}
  In the paper we present new stability and optimal error analyses of hybridized discontinuous Galerkin (HDG) methods which do not require elliptic regularity assumptions. To obtain error estimates without elliptic regularity assumptions, we use new inf-sup conditions based on stabilized saddle point structures of HDG methods. We show that this approach can be applied to obtain optimal error estimates of HDG methods for the Poisson equations, the convection-reaction-diffusion equations, the Stokes equations, and the Oseen equations.
\end{abstract}

\section{Introduction}
\label{sec:introduction}
%

The hybridization (or static condensation) idea in the theory of finite element methods was introduced to reduce computational costs of mixed methods \cite{veubeke}. It was also discovered that the hybridization can be useful to obtain better numerical solutions via post-processing \cite{Arnold:1985}. Then the hybridization technique was employed to discontinuous Galerkin methods \cite{Cockburn:2008} to mitigate the high computational cost of discontinuous Galerkin methods. 
In hybridized discontinuous Galerkin (HDG) methods, a trace unknown residing on the skeleton of meshes is introduced in addition to the original unknowns. After static condensation, the linear system can be reduced to a linear system such that the trace unknown is the only globally coupled unknown, therefore the system size is substantially reduced.
After a unified hybridization framework was introduced in \cite{Cockburn:2009}, HDG methods have been applied to various partial differential equations including the Poisson equations \cite{Cockburn:2008,Cockburn:2010}, the Helmholtz equations \cite{Griesmaier:2011,Cui:2013}, the convection-diffusion equations \cite{Chen-Cockburn:2012,Fu-Qiu-Zhang:2015}, the Stokes and the Oseen equations \cite{Cockburn-Gopalakrishnan:2011,Nguyen-Cockburn:2010,Cockburn-Shi:2013,Cockburn-Sayas:2014,Cockburn-Cui:2012,Cockburn-Cesmelioglu:2013}, and the Maxwell equations \cite{Chen-Qiu:2017,Lu-Chen:2017}, to name a few.


In most error analysis results of HDG methods with dual-mixed formulations, superconvergent error estimates of primal unknowns are obtained by a duality argument under the full elliptic regularity assumption.
However, the elliptic regularity assumption is not available, for instance, for domains with general non-convex geometry and for partial differential equations with discontinuous or sign-changing coefficients. Without the elliptic regularity assumption, optimal error estimates of primal unknowns are not clear. From this point of view, the error analysis of HDG methods relying on the regularity assumption is an obstacle to extend HDG methods to general problems with theoretical support of error analysis. 

We remark that there are a few optimal error estimate results of HDG methods which do not rely on an elliptic regularity assumption. In \cite{Cockburn-Dubois:2014}, an error estimate for the Poisson equations was obtained by using the recovery operator from the trace variable to the pressure variable. However, its extension to non-symmetric problems, such as the convection-diffusion equations, is not obvious. In \cite{Fu-Qiu-Zhang:2015}, an error analysis for the convection-diffusion equations was presented but the analysis requires a restrictive assumption on convection velocity. To the best of our knowledge, except these two equations, a certain form of elliptic regularity assumption is required for error analyses of the HDG methods based on dual-mixed formulations. 

The purpose of this paper is to present a new approach to obtain optimal error estimates of HDG methods without elliptic regularity assumptions. The main idea is to utilize a stabilized saddle point structure of HDG methods and derive a Babu\v{s}ka--Aziz type inf-sup condition of the system. This analysis was developed in \cite{Lee-Rhebergen} for the Poisson-type equations with sign-changing coefficients in which the elliptic regularity assumption is not true even on convex domains.  
In this paper we show that the idea is useful to analyze HDG methods for other partial differential equations without elliptic regularity assumptions. 
For this, we first show the stability and error analyses for the Poisson equations and the convection-reaction-diffusion equations in this paper. 
Then we develop the idea further to cover the Stokes equations and the Oseen equations. 
To the best of our knowledge, this is the first analysis result of HDG methods for these two equations without elliptic regularity assumptions.

The paper is organized as follows. In Section~\ref{sec:notation} we define symbols and notation in the paper. In Section~\ref{sec:poisson} we show the stability and a priori error estimate of the Poisson equations without the elliptic regularity assumption. In Section~\ref{sec:convection-diffusion}, we extend the analysis to the convection-diffusion-reaction equations. In Sections~\ref{sec:stokes} and \ref{sec:oseen}, the stability and a priori error estimates are developed for the Stokes and Oseen equations, respectively. We summarize the results in Section~\ref{sec:conclusion} with concluding remarks.

\section{Notation and definitions} \label{sec:notation}
Let $\Omega \subset \R^d$ with $d = 2,3$, be a bounded domain with polygonal or polyhedral boundary. Let $\mc{T}_h$ be a conforming triangulation of $\Omega$, i.e., $\mc{T}_h$ is a set of closed $d$-dimensional simplices whose interiors are disjoint such that $\cup_{K \in \mc{T}_h} K = \overline{\Omega}$. We use $\mathcal{F}_h$ to denote the set of closed $(d-1)$-dimensional simplices of the triangulation $\mc{T}_h$. For $K \in \mc{T}_h$ and $F \in \mc{F}_h$, $h_K$ and $h_F$ are the diameters of $K$ and $F$, respectively. For given triangulation $\mc{T}_h$, $h$ denotes $\max_{K \in \mc{T}_h} h_K$. We assume that the family of triangulations in the paper satisfies the shape regularity property. As a consequence, there exist uniform constants $C_1$ and $C_2$ independent of $h$ such that $h_K \le C_1 h_F$ and $h_F \le C_2 h_K$ for any $K \in \mc{T}_h$ and $F \in \mc{F}_h$ such that $F \subset \pd K$. 

For a set $D \subset \R^d$ we use $L^2(D)$ to denote the space of square-integrable functions on $D$ with the Lebesgue measure of $D$. For a finite dimensional linear space $\Bbb{X}$ on $\R$, $L^2(D; \Bbb{X})$ is the space of $\Bbb{X}$-valued square-integrable functions with the Euclidean inner product on $\Bbb{X}$. We use $\LRp{ \cdot, \cdot}_D$ and $\LRa{\cdot, \cdot}_D$ to denote the $L^2$ inner product on $L^2(D)$ when $D$ is a union of $d$-dimensional or $(d-1)$-dimensional simplices. We will use the same symbols for the inner products on $L^2(D; \Bbb{X})$. For $K \in \mc{T}_h$ and $p, q \in L^2(\pd K)$, $\LRa{ p, q}_{\pd K}$ stands for the integral $\int_{\pd K} p q \,\dif s$. For simplicity, we use $\LRp{ \cdot, \cdot}$ for 
$\LRp{\cdot, \cdot}_{\Omega}$. Similarly, $\LRa{\cdot,\cdot} = \sum_{K \in \mc{T}_h} \LRa{ \cdot, \cdot}_{\pd K}$ but we will use $\LRa{\cdot, \cdot}_{\pd \mc{T}_h}$ instead of $\LRa{\cdot, \cdot}$ when we need to specify the domain of integration.  

For $K \in \mc{T}_h$, $\nb_K$ is the unit outward normal vector field on $\pd K$, the boundary of $K$. For functions $\vb \in L^2(K; \R^d)$ and $q \in L^2(K)$ such that $\vb \cdot \nb_K$ and $q$ are well-defined 
on $\partial K$, we define $\langle \vb \cdot \nb, q \rangle_{\partial K} := \int_{\partial K} \vb \cdot \nb_K q \dif s$. By the aforementioned convention, 
\algns{
  \LRa{ \vb \cdot \nb, q} := \sum_{K \in \mc{T}_h} \LRa{ \vb \cdot \nb, q}_{\pd K} .
}
%

\section{The Poisson problems}
\label{sec:poisson}
We assume that $\Gamma_D$ and $\Gamma_N$ are disjoint subsets of $\partial \Omega$ such that $\partial \Omega = \overline{\Gamma_D} \cup \overline{\Gamma_N}$. We define $\mathcal{F}_h^D$ as the subset of
$\mathcal{F}_h$ such that $F \subset \overline{\Gamma_D}$ if $F \in \mc{F}_h^D$. $\mc{F}_h^N$ is defined similarly. We assume that all triangulations are conforming to $\Gamma_D$ and $\Gamma_N$. In other words, $\overline{\Gamma_D} = \cup_{F \in \mathcal{F}_h^D} F$ and
$\overline{\Gamma_N} = \cup_{F \in \mathcal{F}_h^N} F$. For simplicity, we assume that $\Gamma_D \not = \emptyset$ in the rest of this paper.

Let $\kappab = \kappab(x)$ be a tensor field on $\Omega$ which is symmetric positive definite at almost every $x \in \Omega$. Throughout this paper we assume that 
\begin{align}
  \label{eq:sigma-min}
  0 < \kappa_{\min} |\xi|^2 \le  \xi^T \kappab(x) \xi  \le \kappa_{\max} |\xi|^2 < +\infty
\end{align}
for all $0 \not = \xi \in \R^d$ and almost every $x \in \Omega$. 
The Poisson equation is to seek a function $p : \Omega \to \mathbb{R}$ such that 
\begin{subequations}
  \label{eq:problem}
  \begin{align}
    \div \LRp{\kappab \grad p} &= f && \text{ in } \Omega, \\
    p &= p_D && \text{ on } \Gamma_D, \\
    - \kappab \grad p \cdot \nb &= p_N && \text{ on } \Gamma_N,
  \end{align}
\end{subequations}
where $p_D \in H^{\frac 12} (\Gamma_D)$ and $p_N \in H^{-\frac 12}( \Gamma_N)$ are given boundary data and $f \in L^2(\Omega)$ is a given source function.

\subsection{HDG methods for the Poisson equation}

In the mixed formulation of the Poisson equation we introduce an auxiliary variable $\ub = -\kappab \grad p$ and rewrite \eqref{eq:problem} as a system of first order equations:
\begin{subequations}
  \label{eq:mixed-poisson}
  \begin{align}
    \kappab^{-1} \ub  + \grad p &= 0 && \text{ in } \Omega, \\
    - \div \ub &= f && \text{ in } \Omega, \\
    p &= p_D && \text{ on } \Gamma_D, \\
    \ub \cdot \nb &= p_N && \text{ on } \Gamma_N.
  \end{align}
\end{subequations}
To define HDG methods for \eqref{eq:mixed-poisson} we will use the following finite element spaces:
\begin{align}
\label{eq:poisson-Vh}  \Vb_h &= \LRb{ \vb \in L^2(\Omega; \R^d) \;:\; \vb|_{K} \in \Vb(K) , \quad \forall K \in \mathcal{T}_h }, \\
\label{eq:poisson-Qh}  Q_h &= \LRb{ v \in L^2(\Omega) \;:\; v|_{K} \in Q(K) , \quad \forall K \in \mathcal{T}_h }, \\
\label{eq:poisson-Mh}  M_{h} &= \LRb{ \mu \in L^2(\mathcal{F}_h) \;:\; \mu|_{F} \in M(F) , \quad \forall F \in \mathcal{F}_h, \quad \mu|_{\mathcal{F}_h^D} = 0 }, 
\end{align}
where $\Vb(K)$, $Q(K)$, $M(F)$ are finite dimensional spaces on the domains $K$ and $F$.
There are many versions of HDG methods up to the choices of $\Vb(K)$, $Q(K)$, $M(F)$. The most common HDG method in the literature uses 
\algn{ \label{eq:poisson-spaces}
\Vb(K) = \mc{P}_k (K; \R^d), \qquad Q(K) = \mc{P}_k (K), \qquad M(F) = \mc{P}_k(F), \qquad k \ge 0,
}
because it gives optimal equal order of convergence for $\ub$ and $p$, and a superconvergence result of $p$ can be obtained by a duality argument under the full elliptic regularity assumption. Therefore, we will only consider the HDG method with \eqref{eq:poisson-spaces}. 


For $(\ub, p)$, a solution of \eqref{eq:mixed-poisson}, let $\pbar$ be the restriction of $p$ on $\mc{F}_h \setminus \mc{F}_h^D$ and 
$\pbar = 0$ on $\mc{F}_h^D$. 
By the integration by parts, if $(\ub, p) \in H^1(\Omega; \R^d) \times H^1(\Omega)$, then it satisfies 
\begin{subequations}
  \label{eq:poisson-eqs}
  \begin{align}
    \label{eq:poisson-eq1}
    \LRp{ \kappab^{-1} \ub, \vb }
    + \LRp{ \grad p, \vb}
    - \LRa{ p - \pbar, \vb \cdot \nb }
    &= - \LRa{p_D, \vb \cdot \nb}_{\mathcal{F}_h^D}  ,
    \\
    \label{eq:poisson-eq2}
    \LRp{ \ub, \grad q}
    - \LRa{ \ub \cdot \nb + \tau (p - \pbar), q }
    &= - \LRa{ \tau p_D, q}_{\mathcal{F}_h^D} + \LRp{ f, q} ,
    \\
    \label{eq:poisson-eq3}
    \LRa{ \ub \cdot \nb + \tau (p - \pbar), \qbar}_{\mc{F}_h \setminus \mathcal{F}_h^D}
    &= \LRa{ p_N, \qbar }_{ \mathcal{F}_h^N} ,
  \end{align}
\end{subequations}
for any $(\vb, q, \qbar) \in \Vb_h \times Q_h \times M_h$. 
Here $\tau$ is a piecewise constant function with positive values on $\mc{F}_h$. 
In fact, $p - \pbar$ vanishes on $\mc{F}_h \setminus \mc{F}_h^D$ in the above equations, 
so all terms including $p - \pbar$ in the above formula vanish for any $\tau$. 
However, we keep those terms here for comparison with the HDG formulation below.


In HDG methods we consider a discrete version of \eqref{eq:poisson-eqs}: Find
$(\ub_h, p_h, \pbar_h) \in \Vb_h \times Q_h \times M_h$ such that
\begin{subequations}
  \label{eq:poisson-HDG-eqs}
  \begin{align}
    \label{eq:poisson-HDG-eq1}
    \LRp{ \kappab^{-1} \ub_h, \vb }
    + \LRp{ \grad p_h, \vb}
    - \LRa{ p_h - \pbar_h, \vb \cdot \nb }
    &= - \LRa{p_D, \vb \cdot \nb}_{\mathcal{F}_h^D}  ,
    \\
    \label{eq:poisson-HDG-eq2}
    \LRp{ \ub_h, \grad q}
    - \LRa{ \ub_h \cdot \nb + \tau (p_h - \pbar_h), q }
    &= - \LRa{ \tau p_D, q}_{\mathcal{F}_h^D} + \LRp{ f, q} ,
    \\
    \label{eq:poisson-HDG-eq3}
    \LRa{ \ub_h \cdot \nb + \tau (p_h - \pbar_h), \qbar}_{\mc{F}_h \setminus \mathcal{F}_h^D}
    &= \LRa{ p_N, \qbar }_{ \mathcal{F}_h^N} ,
  \end{align}
\end{subequations}
for any $(\vb, q, \qbar) \in \Vb_h \times Q_h \times M_h$. 
We remark that \eqref{eq:poisson-HDG-eqs} is different from the HDG formulations in most other HDG papers. 
We use this formulation here to take the same trial and test function spaces because it is advantageous to reveal a stabilized saddle point structure of the HDG methods. The stabilized saddle point structure will be crucial to obtain an error analysis without the Aubin--Nitsche duality argument. 

For later use we define bilinear forms 
\begin{align} 
  \label{eq:poisson-bilinear-form1}
    a_P(\vb, \vb') &:= \LRp{\kappab^{-1} \vb, \vb'}, & & \vb, \vb' \in \Vb_h, 
    \\
  \label{eq:poisson-bilinear-form2}
    b_P((q, \qbar), \vb) &:=
    \LRp{\grad q, \vb} - \LRa{ q - \qbar, \vb\cdot \nb}, & & \vb \in \Vb_h, q \in Q_h, \qbar \in M_h, 
    \\
  \label{eq:poisson-bilinear-form3}
    c_P( (q, \qbar), (q', \qbar') ) &:=
    \LRa{ \tau (q - \qbar), q' - \qbar' }, & & q, q' \in Q_h, \qbar, \qbar' \in M_h .
\end{align}
Note that 
\mltlns{
\LRa{ \ub_h \cdot \nb + \tau( p_h - \pbar_h), q} - \LRa{ \ub_h \cdot \nb + \tau( p_h - \pbar_h), \qbar}_{\pd \mc{T}_h \setminus \mc{F}_h^D} \\
=  \LRa{ \ub_h \cdot \nb + \tau( p_h - \pbar_h), q - \qbar}
}
because $\qbar = 0$ on $\mc{F}_h^D$. 
Then the sum of the left-hand sides of \eqref{eq:poisson-HDG-eqs} is written with the above three bilinear forms as 
\algn{
  \label{eq:ah-form}
  \mc{B}_P((\ub_h, p_h, \pbar_h), (\vb, q, \qbar)) &:= a_P(\ub_h, \vb) + b_P((p_h, \pbar_h), \vb) \\
  \notag &\quad + b_P((q, \qbar), \ub_h) - c_P((p_h, \pbar_h), (q, \qbar)) .
}

For simplicity we assume $p_D = p_N = 0$ in the remainder of this paper but the arguments below can be easily extended to the cases with inhomogeneous boundary conditions.

%
\begin{lemma} 
  \label{lemma:div-surjective}
  For $q \in Q_h$, there exists $\wb \in H^1(\Omega; \Bbb{R}^d)$ such that $\div \wb = q$, $\wb \cdot \nb = 0$ on $\Gamma_N$, and $\norm{ \wb }_{1} \leq C_{\Omega} \norm{ q }_{0}$ with $C_{\Omega}$ which only depends on $\Omega$ and $\Gamma_D$.
\end{lemma}
\begin{proof}
This is a known result (see, e.g., \cite[p.176]{Galdi:2011}), so proof is omitted.
\end{proof}

The norm $\norm{\vb}_{\Vb_h}$ for $\vb \in \Vb_h$ and semi-norm $| q - \qbar |_{\tau, s, \mc{F}_h}$ for $(q, \qbar) \in Q_h \times M_h$ are 
\algn{ \label{eq:poisson-seminorm}
  \norm{ \vb }_{\Vb_h}^2 = \LRp{\kappab^{-1} \vb, \vb}, \qquad 
| q - \qbar |_{\tau, s, \mc{F}_h}^2 = \sum_{K \in \mc{T}_h} h_K^{2s} \LRa{ \tau \LRp{q - \qbar}, q - \qbar}_{\pd K} ,
}
and $\Xb_h$ is the space $\Vb_h \times Q_h \times M_h$ with the norm 
\begin{align*}
\norm{(\vb, q, \qbar)}_{\Xb_h}^2 = \norm{\vb}_{\Vb_h}^2 + \norm{q}_0^2 + | q - \qbar |_{\tau, 0, \mc{F}_h}^2 .
\end{align*}
We will show that $\mc{B}_P$ satisfies an inf-sup condition, which will be used for optimal error estimates without elliptic regularity assumptions.
As the first step for the inf-sup condition of $\mc{B}_P$, we need a weak inf-sup condition stated below.
\begin{lemma} \label{lemma:poisson-weak-inf-sup}
Suppose that $\tau \ge \tau_{\min} $ on $\mc{F}_h$ for a constant $\tau_{\min} >0$.
Then for $b_P(\cdot, \cdot)$ in \eqref{eq:poisson-bilinear-form2}, there exist
  $C_0, C_1 >0$ independent of $h$ such that
  \begin{equation}
    \label{eq:poisson-weak-inf-sup}
    \inf_{(q, \qbar) \in Q_h \times M_h} \sup_{\vb \in \Vb_h }
    \frac{b_P((q, \qbar), \vb)}{\norm{ \vb }_{0} }
    > C_0 \norm{ q }_{0} - \frac{C_1  }{\tau_{\min}}  | q - \qbar |_{\tau, \frac 12, \mc{F}_h} .
  \end{equation}
\end{lemma}
\begin{proof}
  We show \eqref{eq:poisson-weak-inf-sup} by proving an equivalent condition, i.e., there exist
  $C_1', C_2' >0$ independent of $h$ such that for any
  $0 \neq (q, \qbar) \in Q_h \times M_h$ one can find 
  $\vb \in \Vb_h$ such that
  $\norm{ \vb }_{0} \le C_2' \norm{ q }_{0}$ and
  \begin{align}
    \label{eq:poisson-weak-inf-sup-new}
    b_P((q, \qbar), \vb)
    \ge \norm{ q }_{0}^2 - \frac{C_1'}{\tau_{\min}} | q - \qbar |_{\tau, \frac 12, \mc{F}_h} \norm{ q }_{0}.
  \end{align}
  %
  
  To prove \eqref{eq:poisson-weak-inf-sup-new} we first note that there exists
  $\wb \in H^1(\Omega; \Bbb{R}^d)$ such that $\wb \cdot \nb|_{\Gamma_N} = 0$, 
  $\div \wb = -q$, and
  $\norm{ \wb }_{1} \le C_{\Omega} \norm{ q }_{0}$ with
  $C_{\Omega}$ depending only on $\Omega$ and $\Gamma_D$  
  by Lemma~\ref{lemma:div-surjective}. Let $\Pi_h \wb$ be the $L^2$ projection of $\wb$ into $\Vb_h$. Then 
  \begin{equation}
    \label{eq:bhuhw}
    \begin{split}
      -\norm{ q }_{0}^2 
      &= \LRp{q, \div \wb} \\
      &= - \LRp{ \grad q, \wb} + \LRa{ q, \wb \cdot \nb } \\
      &= - \LRp{ \grad q, \Pi_h \wb} + \LRa{ q - \qbar, \wb \cdot \nb }  \\
      &= - b_P \LRp{(q, \qbar), \Pi_h \wb} + \LRa{ q - \qbar, (\wb - \Pi_h \wb ) \cdot \nb } ,
    \end{split}
  \end{equation}
  where the third equality is obtained by applying the facts that $\wb\cdot\nb$ and $\qbar$ are single-valued on $\mc{F}_h$, and the facts that $\qbar = 0$ on $\Gamma_D$ and $\wb \cdot \nb = 0$ on $\Gamma_N$.
    
  An element-wise trace inequality gives
  \begin{equation}
    \label{eq:traceineqangles}
    \begin{split}
     \left| \LRa{ q - \qbar, (\wb - \Pi_h \wb ) \cdot \nb } \right|
     & \le \frac{C}{\tau_{\min}} | q - \qbar |_{\tau, \frac 12, \mc{F}_h} \norm{ \wb }_{1}
     \\
     & \le \frac{C C_{\Omega}}{\tau_{\min}} | q - \qbar |_{\tau, \frac 12, \mc{F}_h } \norm{ q }_{0}      
    \end{split}
  \end{equation}
  with $C$ depending on the implicit constant in the trace inequality.
  Combining \eqref{eq:bhuhw} and \eqref{eq:traceineqangles} we obtain
  \begin{equation}
    b_P\LRp{(q, \qbar), \Pi_h \wb} \ge \norm{ q }_{0}^2
    - \frac{C C_{\Omega}}{\tau_{\min}} | q - \qbar |_{\tau, \frac 12, \mc{F}_h } \norm{ q }_{0}.
  \end{equation}
  Moreover, $\norm{ \Pi_h \wb }_{0} \le \norm{ \wb }_{0} \le \norm{ \wb }_{1} \le C_{\Omega} \norm{ q }_{0}$, so the assertion follows.
\end{proof}
\begin{remark}
Although we only consider the spaces in \eqref{eq:poisson-spaces}, 
Lemma~\ref{lemma:poisson-weak-inf-sup} can be obtained for other polynomial spaces
if $\grad Q(K) \subset \Vb(K)$ holds. The proof is same.
\end{remark}

We remark that $ | q - \qbar |_{\tau, s, \mc{F}_h} \le h^s | q - \qbar |_{\tau, 0, \mc{F}_h}$ holds. 
We now prove an inf-sup condition of $\mc{B}_P$. 
%
\begin{theorem}[inf-sup condition] \label{thm:poisson-inf-sup}
  If $\tau \ge \tau_{\min} \ge C_{\tau} h^{\frac 12}$ holds in \eqref{eq:poisson-bilinear-form3} on $\mc{F}_h$
  with a constant $C_{\tau}>0$, then there exists $\gamma_P$
  independent of $h$ such that
  %
  %
  \begin{align}
    \label{eq:poisson-Bh-inf-sup}
    \inf_{(\vb, q, \qbar) \in \Xb_h } \sup_{(\vb', q', \qbar') \in \Xb_h }
    \frac{\mc{B}_P((\vb, q, \qbar), (\vb', q', \qbar') )}{\norm{ (\vb, q, \qbar) }_{\Xb_h} \norm{ (\vb', q', \qbar') }_{\Xb_h} } \ge \gamma_P > 0.
  \end{align}
  Here $\gamma_P$ may approach $0$ as $C_{\tau} \ra 0$ or $\kappa_{\min} \ra 0$. 
\end{theorem}
\begin{proof}

  To prove the inf-sup condition we will show the following: There exist constants $C_1, C_2>0$ independent of $h$ such that, 
  for any given $(\vb, q, \qbar) \in \Xb_h$, one can find $(\vb', q', \qbar') \in \Xb_h$ satisfying $\mc{B}_P ((\vb, q, \qbar), (\vb', q', \qbar')) \ge C_1 \norm{(\vb, q, \qbar)}_{\Xb_h}^2$ and $\norm{(\vb', q', \qbar')}_{\Xb_h} \le C_2 \norm{(\vb, q, \qbar)}_{\Xb_h}$.
  
  Let $(\vb, q, \qbar) \in \Xb_h$ be given and
  suppose that $\vb_0 \in \Vb_h$ is an
  element satisfying \eqref{eq:poisson-weak-inf-sup-new}  with $\norm{\vb_0}_{0} \le C_2' \norm{q}_0$ for given $q$ and $\qbar$. 
  We now take
  $\vb' = \vb  + \delta
  \vb_0 $, $q' = - q$,
  $\qbar' = - \qbar$ in \eqref{eq:ah-form}, with
  $\delta > 0$ to be determined later. Then, by \eqref{eq:poisson-weak-inf-sup-new},  
  \algn{
    \notag \mc{B}_P((\vb, q, \qbar), (\vb', q', \qbar'))
     &= \mc{B}_P((\vb, q, \qbar), (\vb, -q, -\qbar)) + \delta \mc{B}_P((\vb, q, \qbar), (\vb_0, 0, 0)) \\
     \notag &= \norm{ \vb }_{\Vb_h}^2 + | q - \qbar |_{\tau, 0, \mc{F}_h}^2 + \delta \LRp{a_P(\vb, \vb_0) + b_P((q, \qbar), \vb_0 ) } \\
    \label{eq:BP-intermediate} &\ge \norm{ \vb }_{\Vb_h}^2 + | q - \qbar |_{\tau, 0, \mc{F}_h}^2 + \delta a_P( \vb, \vb_0)  \\
    \notag &\quad + \delta C_0' \norm{ q }_{0}^2
    -  \delta \frac{C_1'}{\tau_{\min}} | q - \qbar |_{\tau, \frac 12, \mc{F}_h} \norm{ q }_{0}.
  }
  By Young's inequality and the inequality $| \cdot |_{\tau, \frac 12, \mc{F}_h} \le Ch^{\frac 12} | \cdot |_{\tau, 0, \mc{F}_h}$, 
  \algns{
  \delta | a_P (\vb, \vb_0) |
  &\le \frac{1}{2} \norm{\vb}_{\Vb_h}^2 + \frac{\delta^2}{2} \norm{\vb_0}_{\Vb_h}^2 
   \le  \frac{1}{2} \norm{\vb}_{\Vb_h}^2 + \frac 12 \delta^2 \kappa_{\min}^{-1} (C_2')^2 \norm{q}_0^2 , \\
   \delta \frac{C C_1'}{\tau_{\min}} | q - \qbar |_{\tau, \frac 12, \mc{F}_h} \norm{ q }_{0} 
   &\le \frac{ \e_1 h (C C_1')^2}{\tau_{\min}^2}  | q - \qbar |_{\tau, 0, \mc{F}_h}^2 +  \frac{\delta^2 }{4 \e_1 } \norm{q}_0^2 \\
   &\le \frac{ \e_1 (C C_1')^2}{C_{\tau}^2}  | q - \qbar |_{\tau, 0, \mc{F}_h}^2 +  \frac{\delta^2 }{4 \e_1 } \norm{q}_0^2 
  }
  for any $\e_1>0$.
  Using these inequalities to the previous inequality one can obtain 
  \algns{
    \mc{B}_P((\vb, q, \qbar), (\vb', q', \qbar')) &\ge \frac{1}{2} \norm{\vb}_{\Vb_h}^2 + \LRp{1 - \e_1 \frac{(C C_1')^2}{C_{\tau}^2} } | q - \qbar |_{\tau, 0, \mc{F}_h}^2 \\
    &\quad + \delta \LRp{ C_0' - \frac 12 \delta \kappa_{\min}^{-1} (C_2')^2 -  \frac{\delta}{4 \e_1} } \norm{q}_0^2 .
  }
  If we choose sufficiently small $\e_1> 0$, and then 
  choose sufficiently small $\delta$ depending on $\kappa_{\min}$ and $\e_1$,
  we can obtain
  \begin{equation} \label{eq:poisson-Bh-coercive}
    \mc{B}_P(\vb, q, \qbar; \vb', q', \qbar') 
    \ge
    C_1 \LRp{ \norm{ \vb }_{\Vb_h}^2
      + \norm{ q }_{0}^2 + | q - \qbar |_{\tau, 0, \mc{F}_h}^2 }
  \end{equation}
  for some $C_1>0$ which approach $0$ as $C_{\tau} \ra 0$ or $\kappa_{\min} \ra 0$. 
  Finally, from the choice of $(\vb', q', \qbar')$, it is not difficult to derive
  \begin{align} \label{eq:test-bounded}
    \norm{ (\vb', q', \qbar') }_{\Xb_h} \le C_2 \norm{ ( \vb, q, \qbar) }_{\Xb_h},
  \end{align}
  with $C_2$ which depends only on $\delta$.  
  The conclusion follows from \eqref{eq:poisson-Bh-coercive} and \eqref{eq:test-bounded}.
\end{proof}

\subsection{The a priori error estimates}
\label{subsec:error_estimate}

In this subsection we show the a priori error estimates.
We recall that there is an interpolation $\bs{\Pi} = (\Pi_{\Vb}, \Pi_Q):H^1(K; \Bbb{R}^d) \times H^1(K) \ra \Vb (K) \times Q (K)$ defined by 
\algn{
\label{eq:poission-intp1} (\Pi_{\Vb} \vb, \vb')_K &= (\vb, \vb')_K & & \forall \vb' \in \mc{P}_{k-1}(K; \Bbb{R}^d), \\
\label{eq:poission-intp2} (\Pi_Q q, q')_K &= (q, q')_K & & \forall q' \in \mc{P}_{k-1}(K)  \\
\label{eq:poission-intp3} \LRa{ \Pi_{\Vb} \vb \cdot \nb_K + \tau \Pi_Q q , \lambda }_{\pd K} &= \LRa{ \vb \cdot \nb_K + \tau  q , \lambda }_{\pd K} & & \lambda \in \mc{P}_k (\pd K),
}
and it satisfies interpolation error estimates
\algn{
\label{eq:poisson-intp-estm-1} \norm{\vb - \Pi_{\Vb} \vb}_{0,K} &\le Ch_K^{k_{\vb} +1} | \vb |_{k_{\ub}+1, K} + C h_K^{k_q +1} \tau_K^* | q |_{k_q + 1, K}, \\
\label{eq:poisson-intp-estm-2}\norm{q - \Pi_Q q}_{0, K} &\le C h_K^{k_q + 1} | q |_{k_q +1, K} + C \frac{h_K^{k_{\vb}+1}}{\tau_{K}^{\max}} | \div \vb |_{k_{\vb}, K}
}
with $0 \le k_q, k_{\vb} \le k$ and constants $C$ independent of $h_K$ and $\tau$ \cite[Theorem~2.1]{cockburn}. Here, $\tau_K^{\max} = \max_{F \subset \pd K} \tau|_F$ and $\tau_K^* = \max_{F \subset \pd K \setminus F^*} \tau|_F$, where $F^*$ is the face that $\tau_K^{\max}$ is attained, and the implicit constants are independent of $K$ and $\tau$. We remark that \eqref{eq:poisson-intp-estm-1} and \eqref{eq:poisson-intp-estm-2}
give optimal order of approximations when $\tau = O(1)$ on $\pd K$. 
We also define $P_M$ as the $L^2$ projection to $M_h$.

\begin{theorem} Suppose that $\Vb_h$, $Q_h$, $M_h$ are defined by the spaces in \eqref{eq:poisson-spaces} with $k \ge 0$, 
and $\tau$ in \eqref{eq:poisson-HDG-eqs} satisfies the assumption in {\rm Theorem~\ref{thm:poisson-inf-sup}}. 
If $(\ub, p, \pbar)$ and $(\ub_h, p_h, \pbar_h)$ are the solutions of \eqref{eq:mixed-poisson} and  \eqref{eq:poisson-HDG-eqs}, respectively, and $(\ub, p) \in H^1(\Omega; \R^d) \times H^1(\Omega)$, then we have 
\algn{
\label{eq:poisson-estm1} \norm{\ub - \ub_h}_{\Vb_h} &\le 2 \norm{\ub - \Pi_{\Vb} \ub}_{\Vb_h} , \\
\label{eq:poisson-estm2} \norm{p - p_h}_{0} &\le \gamma_P^{-1} \norm{\ub - \Pi_{\Vb} \ub}_{\Vb_h} + \norm{ p - \Pi_Q p }_0 . 
}
%
\end{theorem}
\begin{remark}
The estimate \eqref{eq:poisson-estm1} is already proved in \cite{cockburn} but we include it here for completeness. The proof of \eqref{eq:poisson-estm2} without duality argument is the main contribution of the theorem. An estimate of $P_M p - \pbar_h$ can be obtained with the argument similar to the one in \cite{cockburn}. 
\end{remark}
\begin{proof}
For an unknown $\sigma$ we use $e_{\sigma}$ to denote $\sigma - \sigma_h$, the difference of the exact solution $\sigma$
and its numerical approximation $\sigma_h$. 
Adopting this notation the difference of variational equations \eqref{eq:poisson-eqs} and \eqref{eq:poisson-HDG-eqs}
give error equations 
\begin{subequations}
  \label{eq:poisson-err-eqs}
  \begin{align}
    \label{eq:poisson-err-eq1}
    \LRp{ \kappab^{-1} e_{\ub}, \vb }
    - \LRp{ e_p, \div \vb}
    + \LRa{ e_{\pbar}, \vb \cdot \nb }
    &= 0 ,  && \vb \in \Vb_h,
    \\
    \label{eq:poisson-err-eq2}
    \LRp{ e_{\ub}, \grad q}_{\Omega}
    - \LRa{ e_{\ub} \cdot \nb + \tau (e_p - e_{\pbar}), q }
    &= 0 , && q \in Q_h,
    \\
    \label{eq:poisson-err-eq3}
    \LRa{ e_{\ub} \cdot \nb + \tau (e_p - e_{\pbar}), \qbar}_{\partial \mc{T}_h \setminus \mathcal{F}_h^D}
    &=0 ,
         && \qbar \in M_{h} .
  \end{align}
\end{subequations}
Decomposing the errors as
\algn{
  \label{eq:poisson-err-decomp1} e_{\ub} &= e_{\ub}^I + e_{\ub}^h := \LRp{ \ub - \Pi_{\Vb} \ub} + \LRp{ \Pi_{\Vb} \ub - \ub_h}, \\
  \label{eq:poisson-err-decomp2} e_{p} &= e_{p}^I + e_{p}^h := \LRp{ p - \Pi_{Q} p} + \LRp{ \Pi_{Q} p - p_h}, \\
  \label{eq:poisson-err-decomp3} e_{\pbar} &= e_{\pbar}^I + e_{\pbar}^h := \LRp{ p - P_{M} \pbar} + \LRp{ P_{M} \pbar - \pbar_h},
}
one can observe cancellation properties
\algns{
  \LRp{ e_p^I, \div \vb} = 0, \qquad \LRa{ e_{\pbar}^I, \vb \cdot \nb} = 0, \qquad \LRp{ e_{\ub}^I, \grad q} = 0, \\
  \LRa{ e_{\ub} \cdot \nb + \tau (e_p^I - e_{\pbar}^I), q } = 0, \qquad \LRa{ e_{\ub} \cdot \nb + \tau (e_p^I - e_{\pbar}^I), \qbar } = 0
}
from the definitions of $\Pi_{\Vb}$, $\Pi_Q$, $P_M$, 
for any $\vb \in \Vb_h$, $q \in Q_h$, $\qbar \in M_h$. 
Regarding these reductions, the sum of the equations in \eqref{eq:poisson-err-eqs} results in 
\algn{ \label{eq:BP-id}
\mc{B}_P((e_{\ub}^h, e_p^h, e_{\pbar}^h), (\vb, q, \qbar) ) = - \LRp{ \bs{\kappa}^{-1} e_{\ub}^I, \vb} .
}

If we take $\vb = e_{\ub}^h$, $q = - e_p^h$, $\qbar = -e_{\pbar}^h$ in \eqref{eq:BP-id}, then we get 
\algns{
\norm{ e_{\ub}^h }_{\Vb_h}^2 + |( e_p^h, e_{\pbar}^h) |_{\tau, 0, \mc{F}_h}^2 &= - \LRp{ \bs{\kappa}^{-1} e_{\ub}^I, e_{\ub}^h } 
\le \norm{ e_{\ub}^I }_{\Vb_h} \norm{ e_{\ub}^h }_{\Vb_h} ,
}
so \eqref{eq:poisson-estm1} follows by the triangle inequality.

From the inf-sup condition \eqref{eq:poisson-Bh-inf-sup} there exists $(\vb, q, \qbar) \in \Xb_h$ with $\normw{(\vb, q, \qbar)}{\Xb_h} = 1$ such that 
\algns{ 
\normw{(e_{\ub}^h, e_p^h, e_{\pbar}^h) }{\Xb_h} &\le \LRp{\gamma_P - \e}^{-1} \mc{B}_P((e_{\ub}^h, e_p^h, e_{\pbar}^h), (\vb, q, \qbar) )
}
for any $0 < \e < \gamma_P$. Using \eqref{eq:BP-id}, the Cauchy--Schwarz inequality, and the arbitrariness of $\e$, we find that 
\algns{
\normw{(e_{\ub}^h, e_p^h, e_{\pbar}^h) }{\Xb_h} \le \gamma_P^{-1} | (\kappab^{-1} e_{\ub}^I, \vb) | \le \gamma_P^{-1} \| \ub - \Pi \ub \|_{\Vb_h} ,
}
so \eqref{eq:poisson-estm2} follows by the triangle inequality.

\end{proof}
%

\section{Extension to the Convection-Diffusion-Reaction Equations}
\label{sec:convection-diffusion}
In this section we consider an extension of the analysis for the Poisson equations to a model convection-diffusion-reaction equation.
In the model equation, we find $p$ satisfying
\begin{align*}
    - \bs{\kappa} \lap p  + \betab \cdot \grad p + cp &= f && \text{ in } \Omega, \\
    p &= p_D && \text{ on } \pd \Omega, 
\end{align*}
with the assumtions 
\begin{itemize}
  \item[(A1)] $\bs{\kappa}$ is symmetric positive definite and is constant on $\Omega$  
  \item[(A2)] $\betab \in W^{1,\infty}(\Omega; \R^d)$, $c \in L^\infty(\Omega)$, $f \in L^2(\Omega)$,  $p_D \in H^{\frac 12}(\pd \Omega)$ 
  \item[(A3)] $c_{\betab} := c - \frac 12 \div \betab \ge 0$
\end{itemize}
Taking $\ub = - \bs{\kappa} \grad p$ as an additional unknown, we have a system of equations which finds $\ub$ and $p$ satisfying
\begin{subequations}
  \label{eq:convection-diffusion}
  \algn{
    \bs{\kappa}^{-1} \ub + \grad p &= 0, \\
    \div \ub + \betab \cdot \grad p + cp &= f 
  }
\end{subequations}
with boundary condition $p = p_D$ on $\pd \Omega$.

\subsection{HDG methods for the convection-diffusion-reaction equations}
For HDG methods we define the discrete spaces $\Vb_h$, $Q_h$, $M_h$ as in \eqref{eq:poisson-Vh}, \eqref{eq:poisson-Qh}, \eqref{eq:poisson-Mh} with $\Gamma_D = \pd \Omega$.
As before we take $\pbar$ as the restriction of $p$ on $\mc{F}_h \setminus \mc{F}_h^D$ and $\pbar = 0$ on $\mc{F}_h^D$. 
If the solution $(\ub, p)$ of \eqref{eq:convection-diffusion} is in $H^1(\Omega; \Bbb{R}^d) \times H^1(\Omega)$, then 
by the integration by parts, one can show that the variational equations
\begin{subequations}
  \label{eq:convection-eqs}
  \begin{align}
    \label{eq:convection-eq1}
    &\LRp{ \bs{\kappa}^{-1} \ub, \vb }
    + \LRp{ \grad p, \vb}
    - \LRa{ p - \pbar, \vb \cdot \nb } = - \LRa{p_D, \vb \cdot \nb}_{\mathcal{F}_h^D}  ,
    \\
    \label{eq:convection-eq2}
    &\LRp{ \ub_h + \betab p , \grad q}
    - \LRp{ (c - \div \betab) p, q } \\
    \notag &\quad  - \LRa{ \ub \cdot \nb + (\betab \cdot \nb) \pbar + \tau (p - \pbar), q }  = - \LRp{ \tau p_D , q}_{\mathcal{F}_h^D} - \LRp{ f, q} ,
    \\
    \label{eq:convection-eq3}
    &\LRa{ \ub \cdot \nb + (\betab \cdot \nb) \pbar + \tau (p - \pbar), \qbar}_{\partial \mathcal{T}_h \setminus \mathcal{F}_h^D} = 0 
  \end{align}
\end{subequations}
hold for any piecewise constant function $\tau$ on $\mc{F}_h$ and for all $(\vb, q, \qbar) \in \Vb_h \times Q_h \times M_h$.

An HDG formulation of the convection-diffusion-reaction equation is to seek $(\ub_h, p_h, \pbar_h) \in \Vb_h \times Q_h \times M_h$ satisfying
\begin{subequations}
  \label{eq:convection-HDG-eqs}
  \begin{align}
    \label{eq:convection-HDG-eq1}
    &\LRp{ \bs{\kappa}^{-1} \ub_h, \vb }
    + \LRp{ \grad p_h, \vb}
    - \LRa{ p_h - \pbar_h, \vb \cdot \nb }_{\partial \mathcal{T}_h} = - \LRa{p_D, \vb \cdot \nb}_{\mathcal{F}_h^D}  ,
    \\
    \label{eq:convection-HDG-eq2}
    &\LRp{ \ub_h + \betab p_h , \grad q}
    - \LRp{ (c - \div \betab) p_h, q } \\
    \notag &\quad  - \LRa{ \ub_h \cdot \nb + (\betab \cdot \nb) \pbar_h + \tau (p_h - \pbar_h), q }_{\partial \mathcal{T}_h}  = - \LRp{ \tau p_D , q}_{\mathcal{F}_h^D} - \LRp{ f, q} ,
    \\
    \label{eq:convection-HDG-eq3}
    &\LRa{ \ub_h \cdot \nb + (\betab \cdot \nb) \pbar_h + \tau (p_h - \pbar_h), \qbar}_{\partial \mathcal{T}_h \setminus \mathcal{F}_h^D} = 0 
  \end{align}
\end{subequations}
for all $(\vb, q, \qbar) \in \Vb_h \times Q_h \times M_h$. For the stability of this HDG method we need an additional assumption on $\tau$:
\medskip
\begin{itemize}
  \item[(A4)] $\tau_{\betab} := \tau - \frac 12 \betab \cdot \nb > 0$ on $\mc{F}_h$ 
\end{itemize}
\medskip
We define the sum of the left-hand sides of \eqref{eq:convection-HDG-eqs} by $\mc{B}_C$. Recalling the definition of $\mc{B}_P$ in the previous section, one can see that 
%
\mltln{ 
  \label{eq:convection-diffusion-Bh}   \mc{B}_{C} ((\ub_h, p_h, \pbar_h), (\vb, q, \qbar)) 
 = \mc{B}_P ((\ub_h, p_h, \pbar_h), (\vb, q, \qbar) ) \\
  + \LRp{ \betab p_h, \grad q } - \LRp{ (c - \div \betab) p_h, q} - \LRa{ (\betab \cdot \nb) \pbar_h, q - \qbar } .
}
For the stability analysis we 
define $\tilde{\Xb}_h$ as $\Vb_h \times Q_h \times M_h$ with the norm 
\algn{ \label{eq:convection-Xh}
\norm{(\vb, q, \qbar)}_{\tilde{\Xb}_h}^2 = \norm{\vb}_{\Vb_h}^2 + \norm{q}_0^2 + | q - \qbar |_{\tau_{\betab}, 0, \mc{F}_h}^2 . 
}
In the sequel, $\tau_{\betab, \max}$ and $\tau_{\betab, \min}$ are the essential upper and lower bounds of $\tau_{\betab}$. 

Here we prove an inf-sup condition for the bilinear form $\mc{B}_C$ with $\tilde{\Xb}_h$.
\begin{theorem}    \label{thm:convection-inf-sup}
Suppose that {\rm (A1), (A2), (A3)} and {\rm (A4)} hold, and $\tau_{\betab, \min}$ satisfies $\tau_{\betab, \min} \ge C_{\tau} h^{\frac 12}$ with $C_{\tau}>0$ independent of $h$.  
For $\mc{B}_C$ in \eqref{eq:convection-diffusion-Bh} 
%
  \begin{align}
    \label{eq:convection-Bh-inf-sup}
    \inf_{(\vb, q, \qbar) \in \tilde{\Xb}_h } \sup_{(\vb', q', \qbar') \in \tilde{\Xb}_h }
    \frac{\mc{B}_C((\vb, q, \qbar), (\vb', q', \qbar'))}{\norm{ (\vb, q, \qbar) }_{\tilde{\Xb}_h} \norm{ (\vb', q', \qbar') }_{\tilde{\Xb}_h} } \ge \gamma_C
  \end{align}
holds with $\gamma_C >0$ independent of $h$. 
\end{theorem}
\begin{proof}
As in the proof of Theorem~\ref{thm:poisson-inf-sup} we show that there are two constants $C_1$ and $C_2$ such that 
for any given $(\vb, q, \qbar) \in \tilde{\Xb}_h$ one can find $(\vb', q', \qbar') \in \tilde{\Xb}_h$ satisfying
\algns{
\mc{B}_C ((\vb, q, \qbar), (\vb', q', \qbar')) \ge C_1 \norm{(\vb, q, \qbar)}_{\tilde{\Xb}_h}^2, \quad 
\norm{(\vb', q', \qbar')}_{\tilde{\Xb}_h} \le C_2 \norm{(\vb, q, \qbar)}_{\tilde{\Xb}_h} .
} 

Before the proof, we first claim that 
\mltln{ \label{eq:convection-id-1}
 \LRp{ \betab q, \grad q } - \LRp{ (c - \div \betab) q, q} - \LRa{ (\betab \cdot \nb) \qbar, q - \qbar }  \\
 = - \LRp{ c_{\betab} q, q} + \frac 12 \LRa{ \betab \cdot \nb  (q - \qbar), q - \qbar } 
}
for $q \in Q_h$ and $\qbar \in M_h$.
To see it, note first that the integration by parts gives 
\algns{
  - (\betab q , \grad q ) = - \LRa{ \betab \cdot \nb q , q} + \LRp{ (\div \betab) q , q } + (\betab q , \grad q ),
}
so one can obtain 
\algn{ \label{eq:convection-id-2}
  (\betab q , \grad q ) = \frac 12 \LRa{ \betab \cdot \nb q , q} - \frac 12 \LRp{ (\div \betab) q , q } .
}
In addition, note the identity $\frac 12 \LRa{ \betab \cdot \nb \qbar, \qbar} = 0$ followed by the continuity of $\betab \cdot \nb$ and single-valuedness of $\qbar$ on $\mc{F}_h$. From this identity one can derive another identity
\algn{ \label{eq:convection-id-3}
  \frac 12 \LRa{ \betab \cdot \nb q , q } - \LRa{ (\betab \cdot \nb) \qbar, q - \qbar } = \frac 12 \LRa{ (\betab \cdot \nb) (q - \qbar), q - \qbar } .
}
Then one can show \eqref{eq:convection-id-1} using \eqref{eq:convection-id-2} and \eqref{eq:convection-id-3}. 

Let $(\vb, q, \qbar) \in \tilde{\Xb}_h$ be given. 
By applying a variant of Lemma~\ref{lemma:poisson-weak-inf-sup} such that $\tau$ and $\tau_{\min}$ are replaced by $\tau_{\betab}$ and $\tau_{\betab, \min}$, there is $\vb_0 \in \Vb_h$ satisfying $\norm{ \vb_0 }_{0} \le C_2' \norm{ q }_{0}$ and
\algns{
  b_P((q, \qbar), \vb_0)
  \ge C_0' \norm{ q }_{0}^2 - \frac{C_1'}{\tau_{\betab,\min}} | q - \qbar |_{\tau_{\betab}, \frac 12, \mc{F}_h} \norm{ q }_{0}
}
with positive constants $C_0'$, $C_1'$, $C_2'$ which are independent of $h$ and $(q, \qbar)$. We take $(\vb', q', \qbar') = (\vb + \delta \vb_0, -q, -\qbar)$ with $\delta >0$ which will be determined later.
From \eqref{eq:convection-diffusion-Bh} 
and \eqref{eq:convection-id-1}, one can find that 
\algns{
  \mc{B}_C ((\vb, q, \qbar), (\vb', q', \qbar') ) &= \norm{ \vb }_{\Vb_h}^2 + \delta a_P( \vb, \vb_0) + \delta b_P ((q, \qbar), \vb_0 ) \\
  &\quad + \LRa{ \LRp{\tau - (1/2) \betab \cdot \nb} ( q - \qbar), q - \qbar } + \LRp{ c_{\betab} q, q }   \\
  &\ge \norm{ \vb }_{\Vb_h}^2 + \delta a_P( \vb, \vb_0) + | q - \qbar |_{\tau_{\betab}, 0, \mc{F}_h}^2    \\
  &\quad + \delta C_0' \norm{ q }_{0}^2
  -  \delta \frac{C_1'}{\tau_{\betab,\min}} | q - \qbar |_{\tau_{\betab}, \frac 12, \mc{F}_h} \norm{ q }_{0} .
}
Note that this inequality is completely similar to the last form in \eqref{eq:BP-intermediate} with $\tau_{\betab}$ instead of $\tau$. We omit the rest steps because they are same as the proof of Theorem~\ref{thm:poisson-inf-sup}.
\end{proof}
\subsection{The a priori error estimates}
We show the a priori error estimates of the convection-diffusion-reaction equation.
\begin{theorem} \label{thm:convection-error-estm}
Suppose that $\Vb_h$, $Q_h$, $M_h$ are defined by the spaces in \eqref{eq:poisson-spaces} with $k \ge 0$, 
and the assumption in {\rm Theorem~\ref{thm:convection-inf-sup}} hold. 
If $(\ub, p, \pbar)$ and $(\ub_h, p_h, \pbar_h)$ are the solutions of \eqref{eq:convection-eqs} and \eqref{eq:convection-HDG-eqs}, respectively, and $(\ub, p) \in H^1(\Omega; \R^d) \times H^1(\Omega)$, then 
\mltln{ \label{eq:convection-estm-1}
\norm{ \ub - \ub_h}_{\Vb_h} +  \norm{p - p_h}_0 \\
\le  \LRp{ 1 + \gamma_C^{-1} } \norm{ \ub - \Pi_{\Vb} \ub}_{\Vb_h} + (1 + C_1) \norm{ p - \Pi_Q p }_0  + C_2  h^{\frac 12}  \norm{ \pbar - P_M \pbar }_{\tau_{\betab}, 0, \mc{F}_h} 
}
where
$C_1 = \gamma_C^{-1} \LRp{\norm{c_{\betab}}_{L^\infty(\Omega)} + C\norm{\betab}_{W^{1,\infty}(\Omega)} }$, $C_2 =C \tau_{\betab,\min}^{-\frac 12} \gamma_C^{-1} \norm{\betab}_{W^{1,\infty}(\Omega)}$
with $C>0$ independent of $c$, $\betab$, $\tau_{\betab}$, and $h$.
\end{theorem}
\begin{remark}
If $\ub$ and $p$ have sufficiently high regularities, then \eqref{eq:convection-estm-1} is optimal because $h^{\frac 12}  \norm{ \pbar - P_M \pbar }_{\tau_{\betab}, 0, \mc{F}_h} \le C \tau_{\betab, \max} h^{k+1} \norm{p}_{k+1}$ holds by an inverse trace inequality. 
\end{remark}
\begin{proof}[Proof of Theorem~\ref{thm:convection-error-estm}]
From the differences of \eqref{eq:convection-eqs} and \eqref{eq:convection-HDG-eqs}, the error equations of the convection-diffusion-reaction equation are
\begin{subequations}
  \label{eq:convection-err-eqs}
  \begin{align}
    \label{eq:convection-err-eq1}
    &\LRp{ \bs{\kappa}^{-1} e_{\ub}, \vb }
    + \LRp{ \grad e_p, \vb}
    - \LRa{ e_p - e_{\pbar}, \vb \cdot \nb } = 0   ,     \\
    \label{eq:convection-err-eq2}
    &\LRp{ e_{\ub} + \betab e_p , \grad q}
    - \LRp{ (c - \div \betab) e_p, q } \\
    \notag &\quad  - \LRa{ e_{\ub} \cdot \nb + (\betab \cdot \nb) e_{\pbar} + \tau (e_p - e_{\pbar}), q }  = 0 ,  \\
    \label{eq:convection-err-eq3}
    &\LRa{ e_{\ub} \cdot \nb + (\betab \cdot \nb) e_{\pbar} + \tau (e_p - e_{\pbar}), \qbar}_{\partial \mathcal{T}_h \setminus \mathcal{F}_h^D} = 0 
  \end{align}
\end{subequations}
for all $(\vb, q, \qbar) \in \Vb_h \times Q_h \times M_h$. If we decompose the errors as in \eqref{eq:poisson-err-decomp1}--\eqref{eq:poisson-err-decomp3}, 
then we obtain reduced error equations
\begin{subequations}
  \label{eq:convection-red-err-eqs}
  \begin{align}
    \label{eq:convection-red-err-eq1}
    &\LRp{ \bs{\kappa}^{-1} e_{\ub}^h, \vb }
    + \LRp{ \grad e_p^h, \vb}
    - \LRa{ e_p^h - e_{\pbar}^h, \vb \cdot \nb } = - \LRp{ \bs{\kappa}^{-1} e_{\ub}^I, \vb }   ,     \\
    \label{eq:convection-red-err-eq2}
    &\LRp{ e_{\ub}^h + \betab e_p^h , \grad q} - \LRp{ (c - \div \betab) e_p^h, q } \\
    \notag &\quad  - \LRa{ e_{\ub}^h \cdot \nb + (\betab \cdot \nb) e_{\pbar}^h + \tau (e_p^h - e_{\pbar}^h), q } \\
    \notag &\qquad  = -\LRp{ (\betab - P_0 \betab) e_p^I, \grad q} + \LRp{ (c - \div \betab) e_p^I, q } \\ 
    \notag &\qquad \quad + \LRa{(\betab \cdot \nb - {\mathsf P}_0 \betab \cdot \nb) e_{\pbar}^I, q }  ,  \\
    \label{eq:convection-red-err-eq3}
    &-\LRa{ e_{\ub}^h \cdot \nb + (\betab \cdot \nb) e_{\pbar}^h + \tau (e_p - e_{\pbar}), \qbar}_{\partial \mathcal{T}_h \setminus \mathcal{F}_h^D} \\
    \notag &\qquad = -\LRa{(\betab \cdot \nb - {\mathsf P}_0 \betab \cdot \nb) e_{\pbar}^I, \qbar }_{\partial \mathcal{T}_h \setminus \mathcal{F}_h^D} 
  \end{align}
\end{subequations}
where $P_0$ and ${\mathsf P}_0$ are the $L^2$ projections into the spaces of piecewise constant functions on $\mc{T}_h$ and $\mc{F}_h$, respectively. By adding the equations in \eqref{eq:convection-red-err-eqs}, we get 
\algns{
\mc{B}_C ((e_{\ub}^h, e_p^h, e_{\pbar}^h), (\vb, q, \qbar) ) &=  - \LRp{ \bs{\kappa}^{-1} e_{\ub}^I, \vb } - \LRp{ (\betab - P_0 \betab) e_p^I, \grad q} \\
&\qquad + \LRp{ (c - \div \betab) e_p^I, q }  + \LRa{(\betab \cdot \nb - {\mathsf P}_0 \betab \cdot \nb) e_{\pbar}^I, q - \qbar }  \\
&=: I_1(\vb) + I_2(q) + I_3(q) + I_4(q, \qbar) .
}
By the inf-sup condition \eqref{eq:convection-Bh-inf-sup} there exists $(\vb, q, \qbar) \in \tilde{\Xb}_h$ such that $\norm{(\vb, q, \qbar)}_{\tilde{\Xb}_h} \le 1$ and $\mc{B}_C ((e_{\ub}^h, e_p^h, e_{\pbar}^h), (\vb, q, \qbar) ) \ge (\gamma_C - \e) \norm{(e_{\ub}^h, e_p^h, e_{\pbar}^h)}_{\tilde{\Xb}_h}$ for any $0 < \e < \gamma_C$. If we use the above identity, then we can find 
\algn{ \label{eq:gamma-e-ineq}
\norm{(e_{\ub}^h, e_p^h, e_{\pbar}^h)}_{\tilde{\Xb}_h} &\le (\gamma_C - \e)^{-1} \mc{B}_C ((e_{\ub}^h, e_p^h, e_{\pbar}^h), (\vb, q, \qbar) ) \\
\notag &= (\gamma_C - \e)^{-1} \LRp{ I_1(\vb) + I_2(q) + I_3(q) + I_4(q, \qbar) } .
}
%

We now estimate the terms with the functionals $I_i$, $1 \le i \le 4$.
First, the Cauchy--Schwarz inequality gives 
\algn{ \label{eq:I1-estm}
|I_1 (\vb)| &\le \norm{\vb}_{\Vb_h} \norm{e_{\ub}^I}_{\Vb_h} . 
}
The H\"{o}lder inequality, the estimate $\norm{\betab - P_0 \betab }_{L^\infty(\Omega)} \le C h \norm{\betab}_{W^{1,\infty}(\Omega)}$, and an inverse inequality give 
\algn{ \label{eq:I2-estm}
|I_2 (q)| &\le \norm{\betab - P_0 \betab}_{L^\infty(\Omega)} \norm{e_p^I}_{0} \norm{\grad q}_{0} \le C\norm{\betab}_{W^{1,\infty}(\Omega)} \norm{e_p^I}_{0} \norm{ q}_{0} .
}
The triangle inequality, the H\"{o}lder inequality, the Cauchy--Schwarz inequality give
\algn{ \label{eq:I3-estm}
|I_3 (q)| &\le \left| \LRp{ c_{\betab} e_p^I, q } \right| + \frac 12 | \LRp{ \div \betab e_p^I, q } | \\
\notag &\le \LRp{  \norm{ c_{\betab} }_{L^\infty(\Omega)} + \frac 12 \norm{ \betab }_{W^{1,\infty}(\Omega)} } \norm{e_p^I}_0 \norm{ q }_0 .
}
Finally, the H\"{o}lder inequality and a trace inequality give
\algn{ \label{eq:I4-estm}
|I_4 (q, \qbar )| &\le \norm{ (\betab \cdot \nb - {\mathsf P}_0 \betab \cdot \nb) }_{L^\infty(\mc{F}_h)} \norm{ e_{\pbar}^I }_{1, 0, \mc{F}_h} \norm{ q - \qbar }_{1, 0, \mc{F}_h} \\
\notag &\le C \frac{ h^{\frac 12} }{\tau_{\betab, \min}^{1/2 }} \norm{\betab}_{W^{1,\infty}(\Omega)} \norm{ e_{\pbar}^I }_{1, 0, \mc{F}_h} \norm{ q - \qbar }_{\tau_{\betab}, 0, \mc{F}_h} .
}
Since $\e$ is arbitrary, by applying \eqref{eq:I1-estm}--\eqref{eq:I4-estm} to \eqref{eq:gamma-e-ineq}, one can find that 
\algns{
\norm{(e_{\ub}^h, e_p^h, e_{\pbar}^h)}_{\tilde{\Xb}_h} &\le \gamma_C^{-1} \LRp{  \norm{ e_{\ub}^I}_{\Vb_h} + \LRp{ \norm{c_{\betab}}_{L^\infty(\Omega)} + C\norm{\betab}_{W^{1,\infty}(\Omega)} } \norm{e_p^I}_{0}  } \\
&\quad  + \gamma_C^{-1} {  C \frac{ h^{\frac 12} }{\tau_{\betab, \min}^{1/2}} \norm{\betab}_{W^{1,\infty}(\Omega)} \norm{ e_{\pbar}^I }_{1, 0, \mc{F}_h} } .
}
The estimate \eqref{eq:convection-estm-1} follows by this estimate and the triangle inequality.
\end{proof}

\section{The Stokes equations}
\label{sec:stokes}
In this section we will present an analysis of HDG methods for the Stokes equations using a stabilized saddle point structure. As a consequence, we get optimal error estimates of all variables without elliptic regularity assumptions.

For $\bs{f} \in H^{-1}(\Omega; \R^d)$ and a constant $\nu >0$, a model Stokes problem for viscous incompressible Newtonian fluids is the boundary value problem
\algn{
\label{eq:stokes-1} - \div (\nu \grad \ub - p\, \Bbb{I}) &= \bs{f} & &  \text{ in } \Omega, \\
\label{eq:stokes-2} \div \ub &= 0 & & \text{ in } \Omega, \\
\label{eq:no-slip}  \ub &= 0 & &  \text{ on } \pd \Omega
}
where $\ub : \Omega \ra \R^d$ is a velocity field of fluid, $p : \Omega \ra \R$ is a pressure field, $\grad \ub$ is the row-wise gradient of $\ub$, and $\div$ in \eqref{eq:stokes-1} is the row-wise divergence operator.
Note that $\int_{\Omega} p \,\dif x = 0$ is a compatibility condition of $p$ from the no-slip boundary condition \eqref{eq:no-slip}. 

\subsection{HDG methods for the Stokes equations}
For HDG methods we introduce a matrix-valued unknown $\sigmab = \nu \grad \ub$ and 
rewrite \eqref{eq:stokes-1} and \eqref{eq:stokes-2} as 
\begin{subequations}
  \label{eq:mixed-stokes-eqs}
  \algn{
    \label{eq:mixed-stokes-eq1} \nu^{-1} \sigmab - \grad \ub &= 0, & & \text{ in } \Omega, \\
    \label{eq:mixed-stokes-eq2} - \div ( \sigmab  - p\, \Bbb{I}) &= \bs{f}, & & \text{ in } \Omega, \\
    \label{eq:mixed-stokes-eq3}  \div \ub&= 0, & & \text{ in } \Omega.
  }
\end{subequations}
%

Here we define finite element spaces for HDG methods.
As in the previous sections, we only consider finite element spaces which give equal order approximations
\begin{align}
\label{eq:stokes-Sigma_h}   \bs{\Sigma}_h &=  \LRb{ \taub \in L^2(\Omega; \R^{d \times d}) \;:\; \taub|_{K} \in \mathcal{P}_{k}(K; \R^{d \times d}) , \quad \forall K \in \mathcal{T}_h }, \\
\label{eq:stokes-V_h}   \Vb_h &= \LRb{ \vb \in L^2(\Omega; \R^d) \;:\; \vb|_{K} \in \mathcal{P}_{k}(K; \R^d), \quad \forall K \in \mathcal{T}_h }, \\
\label{eq:stokes-M_h}   \bs{M}_{h} &= \LRb{ \mu \in L^2(\mathcal{F}_h) \;:\; \mu|_{F} \in \mathcal{P}_{k}(F; \R^d), \quad \forall F \in \mathcal{F}_h, \quad \mu|_{\pd \Omega} = 0 }, \\
\label{eq:stokes-Q_h}   Q_h &= \LRb{ q \in L^2(\Omega) \;:\; \int_{\Omega} q \,\dif x = 0 \quad \text{ and } \quad q|_{K} \in \mathcal{P}_{k}(K) , \quad \forall K \in \mathcal{T}_h } .
\end{align}
%

Let $\ubbar$ be the restriction of $\ub$ on $\mc{F}_h$ for a solution $\ub$ in \eqref{eq:mixed-stokes-eqs}. 
If a solution $(\sigmab, \ub, p)$ of \eqref{eq:mixed-stokes-eqs} is in $H^1(\Omega; \R^{d \times d}) \times H^1(\Omega; \R^d) \times H^1(\Omega)$, then by the integration by parts, it satisfies the variational equations 
\begin{subequations}
\label{eq:stokes-eqs}
\algn{
\label{eq:stokes-eq1} \LRp{ \nu^{-1} \sigmab, \taub} - \LRp{ \grad \ub,  \taub} + \LRa{ \ub - \ubbar, \taub \nb } &= 0, \\
\label{eq:stokes-eq2} - \LRp{ \sigmab, \grad \vb} + \LRp{ p, \div \vb} + \LRa{ \sigmab \nb - p \nb - \mc{S} (\ub - \ubbar), \vb } &= - \LRp{ \bs{f}, \vb}, \\ 
\label{eq:stokes-eq3}  \LRp{ \div \ub, q} - \LRa{ \ub - \ubbar , q \nb} &= 0, \\
\label{eq:stokes-eq4} - \LRa{ \sigmab \nb - p \nb - \mc{S} (\ub - \ubbar), \vbbar } &= 0 
}
\end{subequations}
for a matrix-valued function $\mc{S}$ on $\mc{F}_h$ 
and for any $(\taub, \vb, \vbbar, q) \in \bs{\Sigma}_h \times \Vb_h \times \bs{M}_h \times Q_h$. 
Here we changed the signs of the second and fourth equations to clarify a symmetric structure of this system.

An HDG formulation for the Stokes system (cf. \cite{cockburn}) is to seek $\sigmab_h \in \bs{\Sigma}_h$, $\ub_h \in \Vb_h$, $\ubbar_h \in \bs{M}_h$, $p_h \in Q_h$ such that 
\begin{subequations}
\label{eq:stokes-HDG-eqs}
\algn{
\label{eq:stokes-HDG-eq1} \LRp{ \nu^{-1} \sigmab_h, \taub} - \LRp{ \grad \ub_h,  \taub} + \LRa{ \ub_h - \ubbar_h, \taub \nb } &= 0, \\
\label{eq:stokes-HDG-eq2} - \LRp{ \sigmab_h, \grad \vb} + \LRp{ p_h, \div \vb} + \LRa{ \sigmab_h \nb - p_h \nb - \mc{S} (\ub_h - \ubbar_h), \vb } &= - \LRp{ \bs{f}, \vb}, \\ 
\label{eq:stokes-HDG-eq3}  \LRp{ \div \ub_h, q} - \LRa{ \ub_h - \ubbar_h , q \nb} &= 0, \\
\label{eq:stokes-HDG-eq4} - \LRa{ \sigmab_h \nb - p_h \nb - \mc{S} (\ub_h - \ubbar_h), \vbbar } &= 0 
}
\end{subequations}
holds for any $(\taub, \vb, \vbbar, q) \in \bs{\Sigma}_h \times \Vb_h \times \bs{M}_h \times Q_h$.
Here $\mc{S}$ is a matrix-valued function defined on $\mc{F}_h$ which has a form
\algn{ \label{eq:S-tensor}
  \mc{S} = \tau_n \LRp{ \nb \otimes \nb } + \tau_t \LRp{ \Bbb{I} - \nb \otimes \nb}
}
with piecewise constant $\tau_n, \tau_t > 0$ where $\vb \otimes \wb$ is $\R^{d \times d}$-valued object defined by 
%
\algns{
\vb \otimes \wb = ( \vb_i \wb_j )_{1 \le i, j \le d}, \quad \vb = 
\pmat{ \vb_1 \cdots \vb_d}^T, \quad \wb = \pmat{ \wb_1 \cdots \wb_d}^T 
}
with $A^T$, the transpose of matrix $A$.

Here we define bilinear forms
\algn{
\label{eq:stokes-a} a_S \LRp{ \taub, \taub'} &= \LRp{ \nu^{-1} \taub, \taub'}, \\
\label{eq:stokes-b1} b_{1,S} ( \taub, (\vb, \vbbar)) &= - \LRp{ \grad \vb, \taub } + \LRa{ \vb - \vbbar, \taub \nb} \\
\label{eq:stokes-b2} b_{2,S} ( (\vb, \vbbar), q) &= \LRp{ \div \vb, q} - \LRa{ \vb - \vbbar, q \nb} , \\
\label{eq:stokes-c} c_S( (\vb, \vbbar) , (\vb', \vbbar') ) &= \LRa{\mc{S} (\vb - \vbbar), \vb' - \vbbar' } ,
}
and note that the last terms in  \eqref{eq:stokes-HDG-eq2} and \eqref{eq:stokes-HDG-eq4} can be combined as 
\algns{
 \LRa{ \nu \sigmab_h \nb - p_h \nb - \mc{S} (\ub_h - \ubbar_h), \vb - \vbbar }.
}
Then the sum of the left-hand sides of \eqref{eq:stokes-HDG-eqs} can be written as 
\algn{ \label{eq:Bs}
&\mc{B}_S \LRp{ (\sigmab_h, \ub_h, \ubbar_h, p_h), (\taub, \vb, \vbbar, q) } \\
\notag & \begin{matrix}
:=& a_S(\sigmab_h, \taub) & + b_{1,S} (\taub, (\ub_h, \ubbar_h) ) &  \\
&+ b_{1,S}(\sigmab_h, (\vb, \vbbar)) & & + b_{2,S} ( (\vb, \vbbar), p_h)  \\
& &+ b_{2,S}( (\ub_h, \ubbar_h), q) & - c_S((\ub_h, \ubbar_h), (\vb, \vbbar) ) .
\end{matrix} 
}
We define semi-norms on $\Vb_h \times \bs{M}_h$ as 
\algns{
| \vb - \vbbar |_{\Bbb{I}, s, \mc{F}_h}^2 &:= \sum_{K \in \mc{T}_h} h_K^{2s} \LRa{  \vb - \vbbar, \vb - \vbbar}_{\pd K}, \\
| \vb - \vbbar |_{\mc{S}, s, \mc{F}_h}^2 &:= \sum_{K \in \mc{T}_h} h_K^{2s} \LRa{ \mc{S}( \vb - \vbbar), \vb - \vbbar}_{\pd K}. 
}
%
%
\begin{lemma} \label{lemma:stokes-1}
Suppose that $\tau_n, \tau_t \ge \tau_{\min}$ with a constant $\tau_{\min}>0$ for $\mc{S}$ in \eqref{eq:S-tensor}. 
For $b_{1,S}(\cdot, \cdot)$ in \eqref{eq:stokes-b1} there exist $C_0, C_1>0$ independent of $h$ such that
\algns{
\inf_{(\vb, \vbbar) \Vb_h \times \bs{M}_h } \sup_{\taub \in \bs{\Sigma}_h} \frac{ b_{1,S}(\taub, (\vb, \vbbar)) }{\norm{\taub}_0} \ge C_0 \norm{ \vb }_0 - \frac{C_1}{\tau_{\min}} | \vb - \vbbar |_{\mc{S}, \frac 12, \mc{F}_h} .
}
\end{lemma}
\begin{proof}
This is a vector version of Lemma~\ref{lemma:poisson-weak-inf-sup}. The only required modification is to use the $\mc{S}$-weighted semi-norm of $(\vb, \vbbar)$ instead of the $\tau$-weighted norm. 
With the $\mc{S}$-weighted norm, the same proof of Lemma~\ref{lemma:poisson-weak-inf-sup} can be applied, so we omit details.
\end{proof}
Since \eqref{eq:Bs} has a dual saddle-point problem structure, 
we need an additional lemma for the stability proof. 
\begin{lemma} \label{lemma:stokes-2}
There exists $C_2 >0$ independent of $h$ such that, for any given $q \in Q_h$ one can find $(\vb, \vbbar) \in \Vb_h \times \bs{M}_h$ satisfying
  \algn{ \label{eq:v0-bound}
    b_{2,S}( (\vb, \vbbar), q) = \norm{q}_0^2, \qquad 
    \norm{\grad \vb}_0 + \norm{\vb}_0 + | \vb - \vbbar |_{\Bbb{I}, -\frac 12, \mc{F}_h }   \le C_2 \norm{q}_0 .
  }
%
\end{lemma}
\begin{proof}
By Lemma~\ref{lemma:div-surjective}, there exists $\wb \in H_0^1(\Omega; \R^d)$ such that $\div \wb = q$ and $\norm{\wb}_1 \le C_{\Omega} \norm{q}_0$ with $C_{\Omega}>0$ depending only on $\Omega$. Let $\vb$ and $\vbbar$ be the $L^2$ projections of $\wb$ in $\Vb_h$ and in $\bs{M}_h$, respectively. The integration by parts gives
\algns{
\norm{q}_0^2  &= \LRp{ q, \div \wb } = - \LRp{ \grad q, \wb} + \LRa{ q, \wb \cdot \nb } \\
&= - \LRp{ \grad q, \vb} + \LRa{ q, \vbbar \cdot \nb} =  \LRa{ q, \div \vb} - \LRa{ q \nb, \vb - \vbbar } = b_{2,S}( (\vb, \vbbar), q) .
}

To show the inequality in \eqref{eq:v0-bound}, note that $\norm{\vb}_0 \le \norm{\wb}_0$ holds 
and $\norm{\grad \vb}_0 \le C \norm{\wb}_1$ follows from an inverse inequality and the element-wise Poincar\'{e} inequality. 
Moreover, for a facet $F \subset \pd K$, 
$h_K^{-1} \LRa{ \vb - \vbbar, \vb - \vbbar}_F \le h_K^{-1} \LRa{ \vb - \wb , \vb - \wb }_F \le C \norm{\wb}_{1, K}^2$
holds with a constant $C>0$ depending on the shape regularity and the trace inequality. 
Taking its summation over $F \in \mc{F}_h$ gives $|\vb -\vbbar |_{\mathbb{I}, -\frac 12, \mc{F}_h} \le C \norm{q}_0$, so \eqref{eq:v0-bound} is proved. 
\end{proof}

Let $\Yb_h$ be the space $\bs{\Sigma}_h \times \Vb_h \times \bs{M}_h \times Q_h$ with the norm 
\algn{ \label{eq:stokes-Xh}
\norm{ (\taub, \vb, \vbbar, q) }_{\Yb_h}^2 = \LRp{ \nu^{-1} \taub, \taub} + \norm{\vb}_0^2 + | \vb - \vbbar |_{\mc{S}, 0, \mc{F}_h} + \norm{q}_0^2 . 
}
\begin{theorem} \label{thm:stokes-inf-sup-Bh}
Suppose that $\tau_t, \tau_n \ge \tau_{\min} \ge C_{\tau} h^{\frac 12}$ in \eqref{eq:S-tensor} on $\mc{F}_h$ for a constant $C_{\tau}>0$ independent of $h$. Then 
  \algns{
    \inf_{(\taub, \vb, \vbbar, q) \in \Yb_h} \sup_{(\taub', \vb', \vbbar', q') \in \Yb_h} \frac{ \mc{B}_S \LRp{ (\taub, \vb, \vbbar, q), (\taub', \vb', \vbbar', q') } } {\norm{(\taub, \vb, \vbbar, q) }_{\Yb_h} \norm{(\taub', \vb', \vbbar', q') }_{\Yb_h} } \ge \gamma_S  
  }
holds with $\gamma_S >0$ independent of $h$. 
\end{theorem}
\begin{proof}
Given $(\taub, \vb, \vbbar, q) \in \bs{\Sigma}_h \times \Vb_h \times \bs{M}_h \times Q_h$, there exist $\taub_0 \in \bs{\Sigma}_h$ and $(\vb_0, \vbbar_0) \in \Vb_h \times \bs{M}_h$ such that
\gat{ \label{eq:stokes-tau0-ineq}
  \norm{\taub_0}_0 \le \norm{\vb}_0, \qquad \norm{\grad \vb_0}_0 + \norm{\vb_0}_0 + | \vb_0 - \vbbar_0 |_{\Bbb{I}, -\frac 12, \mc{F}_h }   \le C_2 \norm{q}_0 , \\
\label{eq:tau0-id} b_{1, S}(\taub_0, (\vb , \vbbar) ) \ge \LRp{ C_0 \norm{\vb}_0  - \frac{C_1}{\tau_{\min}} \norm{\vb - \vbbar }_{\mc{S}, \frac 12, \mc{F}_h} } \norm{\taub_0}_0 , \\
\label{eq:v0-id} b_{2, S}( (\vb_0, \vbbar_0), q ) = \norm{q }_0^2 
}
by Lemma~\ref{lemma:stokes-1} and Lemma~\ref{lemma:stokes-2}.
If we take $\taub' = \taub + \delta \taub_0$, $(\vb', \vbbar')  = - (\vb , \vbbar ) + \epsilon (\vb_0, \vbbar_0)$, $q' = q$ in \eqref{eq:Bs} and add the equations altogether, then we get
\algns{
&\mc{B}_S ((\taub , \vb , \vbbar, q), (\taub', \vb', \vbbar', q')) \\
&=
a_S(\taub , \taub ) + \delta a_S(\taub , \taub_0) + b_{1,S} ( \taub , (\vb, \vbbar) ) + \delta  b_{1,S} (\taub_0, (\vb, \vbbar) )&  \\
&\quad  - b_{1,S}( \taub, (\vb, \vbbar) ) + \epsilon b_{1, S}( \taub, (\vb_0, \vbbar_0) ) - b_{2,S} ( (\vb, \vbbar), q) + \epsilon b_{2,S} ( (\vb_0, \vbbar_0), q)  \\
&\quad + b_{2,S}( (\vb, \vbbar), q )  + c_S((\vb, \vbbar), (\vb, \vbbar) ) - \epsilon c_S((\vb, \vbbar), (\vb_0, \vbbar_0) ) \\
&\ge a_S( \taub, \taub) + \delta a_S(\taub, \taub_0)  + \delta  
\LRp{ C_0 \norm{\vb}_0^2  - \frac{C_1}{\tau_{\min}} |\vb - \vbbar |_{\mc{S}, \frac 12, \mc{F}_h} \norm{\vb}_0 }   \\
&\quad + \epsilon b_{1,S} (\nu \taub , (\vb_0, \vbbar_0) + \epsilon \norm{q}_0^2  + c_S((\vb, \vbbar), (\vb, \vbbar) ) - \epsilon c_S((\vb, \vbbar), (\vb_0, \vbbar_0) ) 
}
in which we used \eqref{eq:tau0-id} and \eqref{eq:v0-id} in the last inequality. 
By Young's inequality, we can obtain
\algns{
  \delta \left| a_S (\taub, \taub_0) \right| &\le \frac 14 a_S(\taub , \taub ) + \delta^2 a_S ( \taub_0, \taub_0 ) \le \frac 14 a_S(\taub, \taub ) + \delta^2 \nu^{-1} \norm{ \vb }_0^2  , \\
  \delta \frac{C_1}{\tau_{\min} } | \vb - \vbbar |_{\mc{S}, \frac 12, \mc{F}_h} \norm{\vb}_0 &\le  \delta \frac{ h C_1^2}{2\tau_{\min}^2 C_0} | \vb - \vbbar |_{\mc{S}, 0, \mc{F}_h}^2 + \frac 12 \delta C_0 \norm{\vb}_0^2 , \\
  \epsilon \left| b_{1,S} (\taub; (\vb_0, \vbbar_0)) \right| &\le \frac 14 a_S(\taub , \taub ) + C \epsilon^2 \LRp{ \LRp{ \nu \grad \vb_0, \grad \vb_0} +  | \vb_0 - \vbbar_0 |_{\Bbb{I}, - \frac 12, \mc{F}_h}^2 } \\
&\le \frac 14 a_S (\taub, \taub) + C C_2 \epsilon^2 \max \{ 1, \nu \} \norm{q}_0^2 , \\
  \epsilon \left| c_S((\vb, \vbbar), (\vb_0, \vbbar_0) \right|  &\le \frac 14 | \vb - \vbbar |_{\mc{S}, 0, \mc{F}_h}^2 + \epsilon^2 | \vb_0 - \vbbar_0 |_{\mc{S}, 0, \mc{F}_h}^2 \\
&\le  \frac 14 | \vb - \vbbar |_{\mc{S}, 0, \mc{F}_h}^2 + \bar{\tau} \epsilon^2 | \vb_0 - \vbbar_0 |_{\mc{S}, -\frac 12, \mc{F}_h}^2 \\
&\le  \frac 14 | \vb - \vbbar |_{\mc{S}, 0, \mc{F}_h}^2 + C_2 \bar{\tau} \epsilon^2 \norm{q}_0^2 
}
where 
\algn{ \label{eq:tau-bar}
\bar{\tau} := \max_{K \in \mc{T}_h} \{ \tau_t|_{\pd K} h_K, \tau_n|_{\pd K} h_K \} .
}
If we use these inequalities to the previous inequality of $\mc{B}_S ((\taub , \vb , \vbbar, q), (\taub', \vb', \vbbar', q'))$, 
then we can obtain 
\algn{
\label{eq:Bs-coercive} &\mc{B}_S ((\taub , \vb , \vbbar, q), (\taub', \vb', \vbbar', q')) \\
\notag \quad &\ge \frac 12 a_S ( \taub , \taub ) - \delta^2 \nu^{-1} \norm{\vb}_0^2 + \frac 12 \delta C_0 \norm{\vb}_0^2 - \delta \frac{h C_1^2}{2 \tau_{\min}^2 C_0} | \vb - \vbbar |_{\mc{S}, 0, \mc{F}_h}^2 \\
\notag & \quad + \epsilon \norm{q}_0^2  + | \vb - \vbbar |_{\mc{S}, 0, \mc{F}_h}^2 - (C \max \{ 1, \nu \} +\bar{\tau})  C_2 \epsilon^2 \norm{q}_0^2 .
}
If we choose $\epsilon$ and $\delta$ sufficiently small to satisfy 
\algn{ \label{eq:stokes-ineq-conditions}
\delta \LRp{ \frac 12 C_0 - \delta \nu^{-1} } >0, \qquad 1 - \delta \frac{h C_1^2}{2 \tau_{\min}^2 C_0} >0, \qquad \e \LRp{ 1 - (C \max \{ 1, \nu \} + \bar{\tau} ) \e } >0, 
}
then  
$\mc{B}_S ((\taub , \vb , \vbbar, q), (\taub', \vb', \vbbar', q')) \ge C \norm{(\taub , \vb , \vbbar, q)}_{\Yb_h}^2 $
holds with 
\algn{ \label{eq:stokes-inf-sup-constant}
C = \min \LRc{ \frac 12, \delta \LRp{ \frac 12 C_0 - \delta \nu^{-1} } , 1 - \delta \frac{h C_1^2}{2 \tau_{\min}^2 C_0} , \e \LRp{ 1 - (C \max \{ 1, \nu \} + \bar{\tau} ) \e } } .
}
One can also check that $\norm{(\taub', \vb', \vbbar', q')}_{\Yb_h} \le C \norm{(\taub, \vb, \vbbar, q)}_{\Yb_h}$ with another constant $C$ depending on $\epsilon$ and $\delta$. 
This complete the proof. 
\end{proof}
\begin{remark}
The formula \eqref{eq:stokes-ineq-conditions} in the proof of Theorem~\ref{thm:stokes-inf-sup-Bh} also shows that the inf-sup constant $\gamma_S$ can be independent of $h$ even in the cases that $\tau_t$, $\tau_n$ in \eqref{eq:S-tensor} depend on $h$. More precisely, $\gamma_S$ is independent of $h$ if $h/\tau_{\min}^2$ and $\bar{\tau}$ in \eqref{eq:tau-bar} are uniformly bounded in $h$.
\end{remark}

\subsection{The a priori error estimates} We show the a priori error estimates for the solutions of \eqref{eq:stokes-HDG-eqs}.
Note that there is an interpolation operator $\bs{\Pi} : H^1(K; \R^{d \times d}) \times H^1(K; \R^d) \times H^1(K) \ra \bs{\Sigma}(K) \times \Vb(K) \times Q(K)$, whose components will be denoted by 
$\bs{\Pi} (\taub, \vb, q) = (\Pi_{\bs{\Sigma}} \taub, \Pi_{\Vb} \vb, \Pi_Q q)$ 
satisfying 
\algn{
\label{eq:stokes-intp1} (\Pi_{\bs{\Sigma}} \taub, \taub')_K &= (\taub, \taub')_K & & \forall \taub' \in \mc{P}_{k-1}(K; \Bbb{R}^{d \times d}), \\
\label{eq:stokes-intp2} (\tr \Pi_{\bs{\Sigma}} \taub , \tr \taub')_K &= (\tr \taub, \tr \taub')_K & & \forall \taub' \in \mc{P}_{k}(K; \Bbb{R}^{d \times d}),  \\
\label{eq:stokes-intp3} (\Pi_{\Vb} \vb, \vb')_K &= (\vb, \vb')_K & & \forall \vb' \in \mc{P}_{k-1}(K; \Bbb{R}^d) , \\
\label{eq:stokes-intp4} (\Pi_{Q} q, q')_K &= (q, q')_K & & \forall q' \in \mc{P}_{k-1}(K ) , \\
\label{eq:stokes-intp5} \LRa{ \Pi_{\bs{\Sigma}} \taub \nb_K + \mc{S} \Pi_{\Vb} \vb + \Pi_Q q , \lambda }_{\pd K} &= \LRa{ \taub \nb_K + \mc{S} \vb + q , \lambda }_{\pd K} & & \lambda \in \mc{P}_k (\pd K; \R^d) .
}
It is known that $\bs{\Pi}$ satisfies 
\algns{
\norm{ \vb - \Pi_{\Vb} \vb }_{0,K} &\le C h_K^{k_{\vb} + 1} | \vb |_{k_{\vb} +1, K} + C \frac{ h_K^{k_{\taub, q}+1}}{\max \{ \tau_n, \tau_t \}} | \div \LRp{ \taub - q \Bbb{I} } |_{k_{\taub, q}, K}, \\
\norm{ \taub - \Pi_{\bs{\Sigma}} \taub }_{0,K} &\le C h_K^{k_{\taub}+1} | \taub |_{k_{\taub}+1, K} + C \tau_t \LRp{ \norm{\vb - \Pi_{\Vb} \vb}_{0,K} + h_K^{k_{\vb} +1} | \vb |_{k_{\vb} +1, K} }, \\
\norm{ q - \Pi_Q q }_{0,K} &\le C h_K^{k_q +1} | q |_{k_q +1, K} + \norm{ \taub - \Pi_{\bs{\Sigma}} \taub}_{0,K} + C h_K^{k_{\taub}+1} | \taub|_{k_{\taub} +1, K}
}
where $0 \le k_{\taub}, k_{\vb}, k_q, k_{\taub, q} \le k$ with the assumptions $\tr \taub = 0$ for the last two inequalities and $\div \vb = 0$ for the last inequality (cf. \cite[Theorem~2.3]{cockburn-stokes}). 
The approximation orders of these estimates are optimal when $\tau_t, \tau_n = O(1)$, and further discussions exploiting approximation orders for $h$-dependent $\tau_t$ and $\tau_n$ can be found in \cite{cockburn-stokes}.
We define $P_{\bs{M}} : L^2(\mc{F}_h) \ra \bs{M}_h$ as the $L^2$ projection into $\bs{M}_h$.
\begin{theorem} \label{thm:stokes-err-estm}
Suppose that the assumptions of Theorem~\ref{thm:stokes-inf-sup-Bh} hold. 
If $(\sigmab, \ub, \ubbar, p)$ and $(\sigmab_h, \ub_h, \ubbar_h, p_h)$ are solutions of \eqref{eq:stokes-eqs} and \eqref{eq:stokes-HDG-eqs}, respectively, then 
\algn{
\label{eq:stokes-error-estm-1} \norm{ \sigmab - \sigmab_h }_0 &\le 2 \norm{ \sigmab - \Pi_{\bs{\Sigma}} \sigmab }_0, \\
\label{eq:stokes-error-estm-2} \norm{ \ub - \ub_h }_0 + \norm{ p - p_h }_0 &\le \norm{ \ub - \Pi_{\Vb} \ub}_0 + \norm{ p - \Pi_{Q} p}_0 \\
\notag &\quad  + \gamma_S^{-1} \nu^{-\frac 12} \norm{\sigmab - \Pi_{\bs{\Sigma}} \sigmab }_0 .
}
\end{theorem}
\begin{remark}
\eqref{eq:stokes-error-estm-1} and the estimate of $\norm{p - p_h}_0$ in \eqref{eq:stokes-error-estm-2} are obtained in \cite{cockburn-stokes} without elliptic regularity assumptions. However, the estimate of $\norm{\ub - \ub_h}_0$ in \eqref{eq:stokes-error-estm-2} without elliptic regularity assumptions is a new result. 
\end{remark}
\begin{proof}
From the difference of \eqref{eq:stokes-eqs} and \eqref{eq:stokes-HDG-eqs} 
we get the error equations 
\begin{subequations}
\label{eq:stokes-err}
\algn{
\label{eq:stokes-err-1} \LRp{ \nu^{-1} e_{\sigmab}, \taub} - \LRp{ \grad e_{\ub},  \taub} + \LRa{ e_{\ub} - e_{\ubbar}, \taub \nb } &= 0, \\
\label{eq:stokes-err-2} - \LRp{ e_{\sigmab}, \grad \vb} + \LRp{ e_p, \div \vb} + \LRa{ \nu e_{\sigmab} \nb - e_p \nb - \mc{S} (e_{\ub} - e_{\ubbar} ), \vb } &= 0, \\ 
\label{eq:stokes-err-3}  \LRp{ \div e_{\ub}, q} - \LRa{ e_{\ub} - e_{\ubbar} , q \nb} &= 0, \\
\label{eq:stokes-err-4} - \LRa{ e_{\sigmab} \nb - e_p \nb - \mc{S} (e_{\ub} - e_{\ubbar} ), \vbbar } &= 0 .
}
\end{subequations}
Splitting the errors as
\algns{
e_{\sigmab} &= e_{\sigmab}^I + e_{\sigmab}^h := \LRp{ \sigmab - \Pi_{\bs{\Sigma}}  \sigmab } + \LRp{ \Pi_{\bs{\Sigma}} \sigmab - \sigmab_h }, \\
e_{\ub} &= e_{\ub}^I + e_{\ub}^h := \LRp{ {\ub} - \Pi_{\Vb}  \ub } + \LRp{ \Pi_{\Vb} \ub - \ub_h } , \\
e_{\ubbar} &= e_{\ubbar}^I + e_{\ubbar}^h := \LRp{ {\ub} - P_{\bs{M}}  \ubbar } + \LRp{ P_{\bs{M}} \ub - \ubbar_h } , \\
e_{p} &= e_{p}^I + e_{p}^h := \LRp{ {p} - \Pi_{Q}  p } + \LRp{ \Pi_{Q} p - p_h } ,
} 
the properties of $P_{\bs{M}}$ and $\bs{\Pi}$ in \eqref{eq:stokes-intp1}--\eqref{eq:stokes-intp5} allow us to reduce \eqref{eq:stokes-err} to 
\begin{subequations}
\label{eq:stokes-red-err}
\algn{
\label{eq:stokes-red-err-1} \LRp{ \nu^{-1} e_{\sigmab}^h, \taub} - \LRp{ \grad e_{\ub}^h,  \taub} + \LRa{ e_{\ub}^h - e_{\ubbar}^h, \taub \nb } &= - \LRp{ \nu^{-1} e_{\sigmab}^I, \taub}, \\
\label{eq:stokes-red-err-2} - \LRp{ e_{\sigmab}^h, \grad \vb} + \LRp{ e_p^h, \div \vb} + \LRa{ e_{\sigmab}^h \nb - e_p^h \nb - \mc{S} (e_{\ub}^h - e_{\ubbar}^h ), \vb } &= 0, \\ 
\label{eq:stokes-red-err-3}  \LRp{ \div e_{\ub}^h, q} - \LRa{ e_{\ub}^h - e_{\ubbar}^h , q \nb} &= 0, \\
\label{eq:stokes-red-err-4} - \LRa{ e_{\sigmab}^h \nb - e_p^h \nb - \mc{S} (e_{\ub}^h - e_{\ubbar}^h ), \vbbar } &= 0 
}
\end{subequations}
where we used 
\algns{
- \LRp{ \grad e_{\ub}, \taub} + \LRa{ e_{\ub} - e_{\ubbar}, \taub \nb} &= \LRp{ e_{\ub}, \div \taub} - \LRa{ e_{\ubbar}, \taub \nb } \\
&= \LRp{ e_{\ub}^h, \div \taub} - \LRa{ e_{\ubbar}^h, \taub \nb } \\
&= - \LRp{ \grad e_{\ub}^h, \taub} + \LRa{ e_{\ub}^h - e_{\ubbar}^h, \taub \nb} 
}
in the first equation. 
The sum of these equations give 
\algn{ \label{eq:stokes-err-Bh}
\mc{B}_S ( ( e_{\sigmab}^h, e_{\ub}^h, e_{\ubbar}^h, e_p^h), (\taub, \vb, \vbbar, q)  ) = - \LRp{ \nu^{-1} e_{\sigmab}^I, \taub } .
}
If $\taub = e_{\sigmab}^h$, $\vb = - e_{\ub}^h$, $\ubar = - e_{\ubar}^h$, $q = -e_p^h$ in \eqref{eq:stokes-err-Bh}, then 
\algns{
\LRp{ \nu^{-1} e_{\sigmab}^h, e_{\sigmab}^h } + | e_{\ub}^h - e_{\ubbar}^h |_{\mc{S}, 0, \mc{F}_h}^2 = - \LRp{ \nu^{-1} e_{\sigmab}^I, e_{\sigmab}^h } .
}
By the Cauchy--Schwarz and the triangle inequalities, 
\eqref{eq:stokes-error-estm-1} follows.

We now estimate $\norm{e_{\ub}^h }_0$ and $\norm{ e_p^h }_0$.
By Theorem~\ref{thm:stokes-inf-sup-Bh}, for any $0 < \e < \gamma_S$, there exists $(\taub, \vb, \vbbar, q) \in \Yb_h$ such that 
$\norm{ (\taub, \vb, \vbbar, q) }_{\Yb_h} \le 1$ and 
\algns{
\norm{ \LRp{ e_{\sigmab}^h, e_{\ub}^h, e_{\ubbar}^h, e_p^h } }_{\Yb_h} &\le (\gamma_S - \e)^{-1} \mc{B}_S ( ( e_{\sigmab}^h, e_{\ub}^h, e_{\ubbar}^h, e_p^h), (\taub, \vb, \vbbar, q) ) \\
&= -(\gamma_S - \e)^{-1} \LRp{ \nu^{-1} e_{\sigmab}^I, \taub} .
}
Then 
$\norm{ \LRp{ e_{\sigmab}^h, e_{\ub}^h, e_{\ubbar}^h, e_p^h } }_{\Yb_h} \le (\gamma_S - \e)^{-1} \nu^{-\frac 12} \norm{ \sigmab - \Pi_{\bs{\Sigma}} \sigmab }_0$ by the Cauchy-Schwarz inequality. 
Since $\e$ is arbitrary, this inequality holds for $\e = 0$. Then \eqref{eq:stokes-error-estm-2} follows by the triangle inequality.
\end{proof}
%

\section{Extension to the Oseen equations}
\label{sec:oseen}
We extend the results of the Stokes equations to the Oseen equations.
Throughout this section we assume that $\betab$ is in $H(\div, \Omega)$
with $\div \betab = 0$ and $\betab|_K \in W^{1,\infty}(K; \R^d)$ for all $K \in \mc{T}_h$ with 
\algns{
\norm{\betab}_{W_h^{1,\infty}} := \max_{K \in \mc{T}_h} \norm{ \betab}_{W^{1,\infty}(K; \R^d)} < \infty .
}
The Oseen equation is
\algns{
 - \div ( \nu \grad \ub - \ub \otimes \betab - p\, \Bbb{I}) &= \bs{f}, & & \text{ in } \Omega, \\
 \div \ub&= 0, & & \text{ in } \Omega, 
}
with the boundary condition $\ub = 0$ on $\pd \Omega$. 
Taking $\sigmab = \nu \grad \ub$ as an additional unknown, one can obtain a system of first order equations
\begin{subequations}
\label{eq:oseen-strong}
\algn{
\label{eq:oseen-1} \nu^{-1} \sigmab - \grad \ub &= 0, & & \text{ in } \Omega, \\
\label{eq:oseen-2} - \div ( \sigmab - \ub \otimes \betab - p\, \Bbb{I}) &= \bs{f}, & & \text{ in } \Omega, \\
\label{eq:oseen-3}  \div \ub&= 0, & & \text{ in } \Omega .
}
\end{subequations}
Since $\div \betab = 0$, $\div (\ub \otimes \betab) = (\grad \ub) \betab$ holds 
where $(\grad \ub) \betab$ stands for the matrix-vector product.

\subsection{HDG methods for the Oseen equations}
We define $\bs{\Sigma}_h$, $\Vb_h$, $\bs{M}_h$, $Q_h$ as in \eqref{eq:stokes-Sigma_h}--\eqref{eq:stokes-M_h}.
Letting $\ubbar$ be the restriction of $\ub$ on $\mc{F}_h$, 
one can see by the integration by parts that the solution 
$(\sigmab, \ub, \ubbar, p)$ of \eqref{eq:oseen-strong} with sufficient regularity
satisfies the variational equations
\begin{subequations}
\label{eq:oseen-eqs}
\algn{
\label{eq:oseen-eq1} \LRp{ \nu^{-1} \sigmab, \taub} - \LRp{ \grad \ub,  \taub} + \LRa{ \ub - \ubbar, \taub \nb } &= 0, \\
\label{eq:oseen-eq2} - \LRp{ \sigmab - \ub \otimes \betab, \grad \vb} + \LRp{ p, \div \vb} & \\
\notag + \LRa{ \sigmab \nb - (\betab \cdot \nb) \ub - p \nb - \mc{S} (\ub - \ubbar), \vb } &= - \LRp{ \bs{f}, \vb}, \\ 
\label{eq:oseen-eq3}  \LRp{ \div \ub, q} - \LRa{ \ub - \ubbar , q \nb} &= 0, \\
\label{eq:oseen-eq4}  - \LRa{  \sigmab \nb - (\betab \cdot \nb) \ub - p \nb - \mc{S} (\ub - \ubbar), \vbbar } &= 0 
}
\end{subequations}
for any matrix-valued $\mc{S}$ on $\mc{F}_h$ and any $(\taub, \vb, \vbbar, q) \in \bs{\Sigma}_h \times \Vb_h \times \bs{M}_h \times Q_h$. 

An HDG method for the Oseen equation is to seek $(\sigmab_h, \ub_h, \ubbar_h, p_h) \in \bs{\Sigma}_h \times \Vb_h \times \bs{M}_h \times Q_h $ satisfying 
\begin{subequations}
\label{eq:oseen-HDG-eqs}
\algn{
\label{eq:oseen-HDG-eq1} \LRp{ \nu^{-1} \sigmab_h, \taub} - \LRp{ \grad \ub_h,  \taub} + \LRa{ \ub_h - \ubbar_h, \taub \nb } &= 0, \\
\label{eq:oseen-HDG-eq2} - \LRp{ \sigmab_h - \ub_h \otimes \betab, \grad \vb} + \LRp{ p_h, \div \vb} & \\
\notag + \LRa{ \sigmab_h \nb - (\betab \cdot \nb) \ubbar_h - p_h \nb - \mc{S} (\ub_h - \ubbar_h), \vb } &= - \LRp{ \bs{f}, \vb}, \\ 
\label{eq:oseen-HDG-eq3}  \LRp{ \div \ub_h, q} - \LRa{ \ub_h - \ubbar_h , q \nb} &= 0, \\
\label{eq:oseen-HDG-eq4}  - \LRa{ \sigmab_h \nb - (\betab \cdot \nb) \ubbar_h - p_h \nb - \mc{S} (\ub_h - \ubbar_h), \vbbar } &= 0 
}
\end{subequations}
for all $(\taub, \vb, \vbbar, q) \in \bs{\Sigma}_h \times \Vb_h \times \bs{M}_h \times Q_h$, where $\mc{S}$ is a function of the form \eqref{eq:S-tensor}.  
We assume 
\algn{ \label{eq:S-beta}
\mc{S}_{\betab} := \mc{S} - \frac 12 \betab \cdot \nb \Bbb{I}  \ge \tau_{\betab, n} \nb \otimes \nb + \tau_{\betab, t} (\mathbb{I} - \nb \otimes \nb) 
}
with $\tau_{\betab, n}, \taub_{\betab, t} > 0$ on $\mc{F}_h$, 
and define $\tilde{\Yb}_h$ by $\bs{\Sigma}_h \times \Vb_h \times Q_h \times \bs{M}_h$ 
with the norm 
\algns{
 \norm{(\taub, \vb, \vbbar, q)}_{\tilde{\Yb}_h}^2 := \LRp{ \nu^{-1} \taub, \taub} + \norm{\vb}_0^2 + \norm{q}_0^2 + | \vb - \vbbar |_{\mc{S}_{\betab}, 0, \mc{F}_h}^2 . 
}
Let us define $\mc{B}_O ((\sigmab_h, \ub_h, \ubbar_h, p_h), (\taub, \vb, \vbbar, q)) $ as the sum of all the bilinear forms on the left-hand sides of \eqref{eq:oseen-HDG-eqs}. 
\begin{theorem} \label{thm:oseen-inf-sup-Bh}
Suppose that $\mc{S}$ in \eqref{eq:oseen-HDG-eqs} satisfies \eqref{eq:S-beta} and $\tau_{\betab,t}, \tau_{\betab,n} \ge \tau_{\betab,\min}\ge C_{\tau} h^{\frac 12}$ with $C_{\tau} >0$ independent of $h$.
Then the inf-sup condition 
  \algns{
    \inf_{(\taub, \vb, \vbbar, q) \in \tilde{\Yb}_h} \sup_{(\taub', \vb', \vbbar', q') \in \tilde{\Yb}_h} \frac{ \mc{B}_O \LRp{ (\taub, \vb, \vbbar, q), (\taub', \vb', \vbbar', q') } } {\norm{(\taub, \vb, \vbbar, q) }_{\tilde{\Yb}_h} \norm{(\taub', \vb', \vbbar', q') }_{\tilde{\Yb}_h} } \ge \gamma_O 
  }
holds with $\gamma_O>0$ independent of $h$. 
\end{theorem}
\begin{proof}

Before we begin proof, note an identity for later use. The integration by parts and an algebraic identity give
\algns{
\LRp{ \ub \otimes \betab, \grad \vb} &= \LRa{ (\betab \cdot \nb) \ub, \vb}  - \LRp{  (\grad \ub) \betab, \vb} = \LRa{ (\betab \cdot \nb) \ub, \vb}  - \LRp{  \vb \otimes \betab , \grad \ub} 
}
for any $\ub, \vb \in \Vb_h$ because $\div \betab = 0$.
If $\ub = \vb$, then we can obtain 
\algn{ \label{eq:stokes-identity}
\LRp{ \vb \otimes \betab, \grad \vb} = \frac 12 \LRa{ (\betab \cdot \nb) \vb, \vb} .
}
We also note that, by recalling $\mc{B}_S$ in \eqref{eq:Bs}, one can find that
\algns{
\mc{B}_O ((\taub, \vb, \vbbar, q), (\taub', \vb', \vbbar', q')) &= \mc{B}_S ((\taub, \vb, \vbbar, q), (\taub', \vb', \vbbar', q')) \\
&\quad + \LRp{ \vb \otimes \betab, \grad \vb'} - \LRa{ (\betab \cdot\nb) \vbbar , \vb' - \vbbar'} .
}

We begin the inf-sup condition proof. For given $(\taub, \vb, \vbbar, q) \in \tilde{\Yb}_h$, let $\taub_0 \in \bs{\Sigma}_h$ and $(\vb_0, \vbbar_0) \in \Vb_h \times \bs{M}_h$ be the elements determined by Lemma~\ref{lemma:stokes-1} with $\mc{S}_{\betab}$ instead of $\mc{S}$ and Lemma~\ref{lemma:stokes-2}, which satisfy
\gat{ \label{eq:oseen-tau0-ineq}
  \norm{\taub_0}_0 \le \norm{\vb}_0, \qquad \norm{\grad \vb_0}_0 + \norm{\vb_0}_0 + | \vb_0 - \vbbar_0 |_{\Bbb{I}, -\frac 12, \mc{F}_h }   \le C_2 \norm{q}_0 , \\
\label{eq:oseen-tau0-id} b_{1, S}(\taub_0, (\vb , \vbbar) ) \ge \LRp{ C_0 \norm{\vb}_0  - \frac{C_1}{\tau_{\betab,\min}} \norm{\vb - \vbbar }_{\mc{S}_{\betab}, \frac 12, \mc{F}_h} } \norm{\taub_0}_0 , \\
\label{eq:oseen-v0-id} b_{2, S}( (\vb_0, \vbbar_0), q ) = \norm{q }_0^2 .
}
If we take $\taub' = \taub + \delta \taub_0$, $(\vb', \vbbar') = - (\vb, \vbbar) + \e (\vb_0, \vbbar_0)$, $q' = - q$ with constants $\delta$ and $\e$ which will be determined later, then we have
\algn{
\notag &\mc{B}_O ((\taub, \vb, \vbbar, q), (\taub + \delta \taub_0, -\vb + \e \vb_0 , - \vbbar + \e \vbbar_0, -q ))  \\
\label{eq:Bo-equality} &=\mc{B}_S ((\taub, \vb, \vbbar, q), (\taub + \delta \taub_0, -\vb + \e \vb_0 , - \vbbar + \e \vbbar_0, -q )) \\
\notag &\quad - \LRs{ \LRp{ \vb \otimes \betab, \grad \vb} - \LRa{ (\betab \cdot\nb) \vbbar , \vb - \vbbar} } \\
\notag &\quad + \e \LRs{ \LRp{ \vb \otimes \betab, \grad \vb_0} - \LRa{ (\betab \cdot\nb) \vbbar , \vb_0 - \vbbar_0} } \\
\notag &=: I_1 + I_2 + I_3 .
}

If we use \eqref{eq:stokes-identity} to $I_2$, then we have
\algn{ \label{eq:I2-id}
I_2 &= - \frac 12 \LRa{ (\betab \cdot \nb) \vb, \vb } + \LRa{ (\betab \cdot \nb) \vb, \vb - \vbbar} \\
\notag &= - \frac 12 \LRa{ (\betab \cdot \nb) ( \vb - \vbbar) , \vb - \vbbar} = - | \vb - \vbbar |_{\betab \cdot \nb \Bbb{I}, 0, \mc{F}_h}^2
}
where we used the fact $\frac 12 \LRa{ (\betab \cdot \nb) \vbbar, \vbbar } = 0$ for the second equality, which follows from the continuity of $\betab \cdot \nb$ on $\mc{F}_h$ and the single-valuedness of $\vbbar$ on $\mc{F}_h$.
Note that an inequality corresponding to \eqref{eq:Bs-coercive} with $\mc{S}_{\betab}$ instead of $\mc{S}$, holds for $I_1$. 
Using this inequality and \eqref{eq:I2-id}, one can obtain
\algns{
I_1 + I_2 &\ge \frac 12 a_S ( \taub , \taub ) - \delta^2 \nu^{-1} \norm{\vb}_0^2 + \frac 12 \delta C_0 \norm{\vb}_0^2 - \delta \frac{h C_1^2}{2 \tau_{\betab, \min}^2 C_0} | \vb - \vbbar |_{\mc{S}_{\betab}, 0, \mc{F}_h}^2 \\
\notag & \quad + \epsilon \norm{q}_0^2  + | \vb - \vbbar |_{\mc{S}_{\betab}, 0, \mc{F}_h}^2 - (C \max \{ 1, \nu \} + \bar{\tau}_{\betab})  C_2 \epsilon^2 \norm{q}_0^2 
}
with $\bar{\tau}_{\betab}  = \max_{K \in \mc{T}_h} \{ \tau_{\betab,t}|_{\pd K} h_K, \tau_{\betab, n}|_{\pd K} h_K \}$.

To estimate $I_3$, note the inequalities
\algns{
\left| \LRp{ \vb \otimes \betab , \grad \vb_0 } \right| &\le \norm{\betab}_{L^\infty(\Omega)} \norm{\vb}_0 \norm{\grad \vb_0 }_0 ,\\
\left| \LRa{ (\betab \cdot \nb) \vbbar, \vb_0 - \vbbar_0 } \right| &\le \left| \LRa{ (\betab \cdot \nb) (\vb - \vbbar), \vb_0 - \vbbar_0 } \right| + \left| \LRa{ (\betab \cdot \nb) \vb, \vb_0 - \vbbar_0 } \right| \\
&\le C \norm{ \betab \cdot \nb }_{L^\infty(\mc{F}_h)} \LRp{ \norm{\vb}_0 + | \vb - \vbbar |_{\mathbb{I}, \frac 12, \mc{F}_h} }  | \vb_0 - \vbbar_0 |_{\mathbb{I}, -\frac 12, \mc{F}_h}  
}
where we used the inverse trace inequality $\sum_{K \in \mc{T}_h} \sum_{F \subset \pd K} h_F \norm{\vb}_{L^2(F)}^2 \le C \norm{\vb}_0^2$
in the last inequality. From the above two inequalities and \eqref{eq:oseen-tau0-ineq}  one can obtain
\algns{
\left| I_3 \right| &\le C \e \LRp{ \norm{\vb}_0 + | \vb - \vbbar |_{\mathbb{I}, \frac 12, \mc{F}_h} } \norm{q}_0 \norm{\betab}_{L^\infty(\Omega)} \\
&\le \frac {\e}{2} \norm{q}_0^2 + \frac {\e}2 C^2 \norm{\betab}_{L^\infty(\Omega)}^2 \LRp{ \norm{\vb}_0 + | \vb - \vbbar |_{\mathbb{I}, \frac 12, \mc{F}_h} }^2 .
}
Combining these estimates of $I_1 + I_2$ and $I_3$, with \eqref{eq:Bo-equality}, we have 
\algns{
\notag &\mc{B}_O ((\taub, \vb, \vbbar, q), (\taub + \delta \taub_0, -\vb + \e \vb_0 , - \vbbar + \e \vbbar_0, -q ))  \\
&\ge \frac 12 a_S ( \taub, \taub) - \delta^2 \nu^{-1} \norm{\vb}_0^2 + \frac 12 \delta C_0 \norm{\vb}_0^2 - \delta \frac{h C_1^2}{2 \tau_{\betab, \min}^2 C_0} | \vb - \vbbar |_{\mc{S}_{\betab}, 0, \mc{F}_h}^2 \\
& \quad + \frac{\e}{2} \norm{q}_0^2  + | \vb - \vbbar |_{\mc{S}_{\betab} , 0, \mc{F}_h}^2 - (C \max \{ 1, \nu\} + \bar{\tau}_{\betab})  C_2 \epsilon^2 \norm{q}_0^2 \\
& \quad - \frac{\e}{2} C^2 \norm{\betab}_{L^\infty(\Omega)}^2 \LRp{ \norm{\vb}_0 + | \vb - \vbbar |_{\mathbb{I}, \frac 12, \mc{F}_h} }^2 .
}
If we choose sufficiently small $\delta, \e >0$ such that 
\algns{
 \mc{S}_{\betab} - \delta \frac{h C_1^2}{2 \tau_{\betab, \min}^2 C_0} \mc{S}_{\betab} - \e C^2 h \norm{\betab}_{L^\infty(\Omega)}^2  \mathbb{I} &\ge \frac 12 \mc{S}_{\betab} ,\\
 \delta \LRp{ \frac 12 C_0 - \delta \nu^{-1} } - \e C^2 \norm{\betab}_{L^\infty(\Omega)}^2 &> 0, \\
 \e \LRp{ \frac 12 - \e \LRp{C \max \{ 1, \nu\} + \bar{\tau}_{\betab}} } &> 0 ,
}
then there exists $C>0$ such that 
\algns{
\notag &\mc{B}_O ((\taub, \vb, \vbbar, q), (\taub + \delta \taub_0, -\vb + \e \vb_0 , - \vbbar + \e \vbbar_0, -q ))  
\ge C \norm{(\taub, \vb, \vbbar, q)}_{\tilde{\Yb}_h}^2 .
}
One can also see that 
$\norm{ (\taub', \vb', \vbbar', q') }_{\tilde{\Yb}_h} \le C' \norm{(\taub, \vb, \vbbar, q) }_{\tilde{\Yb}_h}$, 
with $C'$ depending only on $\delta$ and $\e$. This completes the proof.
\end{proof}

\subsection{The a priori error estimates} We show the a priori error estimates for the solutions of \eqref{eq:oseen-HDG-eqs}.

%

\begin{theorem}
Suppose that the assumptions in Theorem~\ref{thm:oseen-inf-sup-Bh} hold. 
If $(\sigmab, \ub, \ubbar, p)$ and $(\sigmab_h, \ub_h, \ubbar_h, p_h)$ are solutions of 
\eqref{eq:oseen-eqs} and \eqref{eq:oseen-HDG-eqs}, respectively, then 
\algn{
\| \sigmab - \sigmab_h \|_0 &\le \| \sigmab - {\Pi}_{\bs{\Sigma}} \sigmab \|_0 + \nu^{\frac 12} \mc{E}_h , \\
\norm{ \ub - \ub_h }_0 &\le \| \ub - {\Pi}_{\Vb} \ub \|_0 + \mc{E}_h , \\
\norm{ p - p_h }_0 &\le \| p - {\Pi}_{Q} p \|_0 + \mc{E}_h 
}
hold with 
\algns{
\mc{E}_h :=&  \gamma_O^{-1} \nu^{-\frac 12}  \| \sigmab - {\Pi}_{\bs{\Sigma}} \sigmab \|_0 + C \gamma_O^{-1} \nu^{-\frac 12} \norm{\betab}_{W_h^{1,\infty}}  \| \ub - {\Pi}_{\Vb} \ub \|_0 \\
&+ \gamma_O^{-1} \tau_{\betab, \min}^{-\frac 12} \norm{\betab}_{W_h^{1,\infty}} h^{\frac 12} \norm{ \ub - P_{\bs{M}} \ub }_{\mathbb{I}, 0, \mc{F}_h}.
}

\end{theorem}
\begin{proof}
The difference of \eqref{eq:oseen-eqs} and \eqref{eq:oseen-HDG-eqs} gives the error equations 
\algns{
\LRp{ \nu^{-1} e_{\sigmab}, \taub} - \LRp{ \grad e_{\ub},  \taub} + \LRa{ e_{\ub} - e_{\ubbar}, \taub \nb } &= 0, \\
 - \LRp{ e_{\sigmab} - e_{\ub} \otimes \betab, \grad \vb} + \LRp{ e_p, \div \vb}  \\
 + \LRa{ e_{\sigmab} \nb  - (\betab \cdot \nb) e_{\ub} - e_p \nb - \mc{S} (e_{\ub} - e_{\ubbar} ), \vb } &= 0, \\ 
 \LRp{ \div e_{\ub}, q} - \LRa{ e_{\ub} - e_{\ubbar} , q \nb} &= 0, \\
 \LRa{ e_{\sigmab} \nb - (\betab \cdot \nb) e_{\ub}  - e_p \nb - \mc{S} (e_{\ub} - e_{\ubbar} ), \vbbar } &= 0 .
}
%

%

Splitting the errors as
\algns{
e_{\sigmab} &= e_{\sigmab}^I + e_{\sigmab}^h := \LRp{ \sigmab - {\Pi}_{\bs{\Sigma}}  \sigmab } + \LRp{ {\Pi}_{\bs{\Sigma}} \sigmab - \sigmab_h }, \\
e_{\ub} &= e_{\ub}^I + e_{\ub}^h := \LRp{ {\ub} - {\Pi}_{\Vb}  \ub } + \LRp{ {\Pi}_{\Vb} \ub - \ub_h } , \\
e_{\ubbar} &= e_{\ubbar}^I + e_{\ubbar}^h := \LRp{ {\ub} - P_{\bs{M}}  \ubbar } + \LRp{ P_{\bs{M}} \ub - \ubbar_h } , \\
e_{p} &= e_{p}^I + e_{p}^h := \LRp{ {p} - {\Pi}_{Q}  p } + \LRp{ {\Pi}_{Q} p - p_h } ,
} 
%
we can reduce \eqref{eq:stokes-err} to 
\begin{subequations}
\label{eq:oseen-red-err}
\algn{
\label{eq:oseen-red-err-1} &\LRp{ \nu^{-1} e_{\sigmab}^h, \taub} - \LRp{ \grad e_{\ub}^h,  \taub} + \LRa{ e_{\ub}^h - e_{\ubbar}^h, \taub \nb }  = - \LRp{ \nu^{-1} e_{\sigmab}^I, \taub}, \\
\label{eq:oseen-red-err-2} & - \LRp{ e_{\sigmab}^h - e_{\ub}^h \otimes \betab, \grad \vb} + \LRp{ e_p^h, \div \vb}  \\
\notag  & + \LRa{ e_{\sigmab}^h \nb  - (\betab \cdot \nb) e_{\ubbar}^h - e_p^h \nb - \mc{S} (e_{\ub}^h - e_{\ubbar}^h ), \vb }  \\ 
\notag  &\quad  = - \LRp{ e_{\ub}^I \otimes \betab , \grad \vb} + \LRa{ (\betab \cdot \nb) e_{\ubbar}^I, \vb}  , \\
\label{eq:oseen-red-err-3} &\LRp{ \div e_{\ub}^h, q} - \LRa{ e_{\ub}^h - e_{\ubbar}^h , q \nb} = 0, \\
\label{eq:oseen-red-err-4} &- \LRa{ e_{\sigmab}^h \nb - (\betab \cdot \nb) e_{\ubbar}^h  - e_p^h \nb - \mc{S} (e_{\ub}^h - e_{\ubbar}^h ), \vbbar }  = - \LRa{ (\betab \cdot \nb) e_{\ubbar}^I, \vbbar} .
}
\end{subequations}
The sum of these equations with \eqref{eq:I2-id} gives
\algns{
&\mc{B}_O \LRp{ (e_{\sigmab}^h, e_{\ub}^h, e_{\ubbar}^h, e_p^h),(\taub, \vb, \vbbar, q)  } \\
&= - \LRp{ \nu^{-1} e_{\sigmab}^I, \taub} - \LRp{ e_{\ub}^I \otimes \betab , \grad \vb} +  \LRa{ (\betab \cdot \nb) e_{\ubbar}^I, \vb - \vbbar} \\
&=: J_1(\taub) + J_2(\vb) + J_3((\vb, \vbbar)) .
}
For any $0 < \e < \gamma_O$, there exists $(\taub, \vb, \vbbar, q) \in \tilde{\Yb}_h$
such that $\norm{(\taub, \vb, \vbbar, q)}_{\tilde{\Yb}_h} \le 1$ and 
\algn{ \label{eq:oseen-coercive}
\norm{(e_{\sigmab}^h, e_{\ub}^h, e_{\ubbar}^h, e_p^h)}_{\tilde{\Yb}_h} &\le (\gamma_O - \e)^{-1} \mc{B}_O \LRp{ (e_{\sigmab}^h, e_{\ub}^h, e_{\ubbar}^h, e_p^h),(\taub, \vb, \vbbar, q)  } \\
\notag &= (\gamma_O - \e)^{-1} \LRp{ J_1(\taub) + J_2(\vb)  + J_3((\vb, \vbbar)) } .
}
From the assumption $\norm{(\taub, \vb, \vbbar, q)}_{\tilde{\Yb}_h} \le 1$, the Cauchy--Schwarz inequality, H\"{o}lder inequality, and a trace inequality, we find that 
\algns{
| J_1(\taub) | &\le \nu^{-\frac 12} \norm{ e_{\sigmab}^I }_0 , \\
| J_2(\vb) | &= | \LRp{ e_{\ub}^I \otimes (\betab - P_0 \betab), \grad \vb } | \le C \norm{\betab}_{W_h^{1,\infty}} \norm{e_{\ub}^I} , \\
| J_3((\vb, \vbbar)) | &= \LRa{ (\betab \cdot \nb - \mathsf{P}_0 \betab \cdot \nb) e_{\ubbar}^I, \vb - \vbbar} \le C \tau_{\betab, \min}^{- \frac 12} \norm{\betab}_{W_h^{1,\infty}(\Omega)} h^{\frac 12} \norm{ e_{\ubbar}^I }_0 .
}
Since \eqref{eq:oseen-coercive} and these three estimates hold for any $\e>0$, combining these results give 
\algns{
\norm{(e_{\sigmab}^h, e_{\ub}^h, e_{\ubbar}^h, e_p^h)}_{\tilde{\Yb}_h} \le \gamma_O^{-1} \LRp{ \nu^{-\frac 12} \norm{e_{\sigmab}^I}_0 + C \norm{\betab}_{W_h^{1,\infty}} \LRp{ \norm{e_{\ub}^I} + \tau_{\betab, \min}^{- \frac 12} h^{\frac 12} \norm{ e_{\ubbar}^I }_0 } }. 
}
The conclusion follows by the triangle inequality.
\end{proof}

\section{Conclusions}
\label{sec:conclusion}
In this paper we discuss the stability and a priori error estimates of HDG methods without the elliptic regularity assumption. An analysis utilizing a stabilized saddle point structure is a key to obtain the results. Based on the idea we showed that optimal error estimates can be obtained for the Poisson, convection-diffusion-reaction, the Stokes, and the Oseen equations. Extension of this analysis to HDG methods for other PDE problems is a topic of future research.

\bibliographystyle{plain} 

\end{document}

%% file: jlmacros.tex
\newcommand{\pmat}[1]{\begin{pmatrix} #1 \end{pmatrix}}

\newcommand{\beq} {\begin{equation}}
\newcommand{\eeq} {\end{equation}}
\newcommand{\bdm} {\begin{displaymath}}
\newcommand{\edm} {\end{displaymath}}

\newcommand{\bit}{\begin{itemize}}
\newcommand{\eit}{\end{itemize}}
\newcommand{\bde}{\begin{description}}
\newcommand{\ede}{\end{description}}

\newcommand{\ben}{\begin{enumerate}}
\newcommand{\een}{\end{enumerate}}
\newcommand{\algn}[1]{\begin{align} #1 \end{align}}
\newcommand{\algns}[1]{\begin{align*} #1 \end{align*}}
\newcommand{\mltln}[1]{\begin{multline} #1 \end{multline}}
\newcommand{\mltlns}[1]{\begin{multline*} #1 \end{multline*}}
\newcommand{\gat}[1]{\begin{gather} #1 \end{gather}}

\newcommand{\barr}{\begin{array}}
\newcommand{\earr}{\end{array}}

\newcommand{\mc}[1]{\mathcal{#1}}
\newcommand{\LRp}[1]{\left( #1 \right)}
\newcommand{\LRb}[1]{\left \{ #1 \right \}}
\newcommand{\LRs}[1]{\left[ #1 \right]}
\newcommand{\LRa}[1]{\left< #1 \right>}

\newcommand{\LRc}[1]{\left\{ #1 \right\}}

\newcommand{\ra}{\rightarrow}

\newcommand{\e}{\epsilon}

\newcommand{\R}{{\mathbb R}}

\newcommand{\grad}{\operatorname{grad}}

\renewcommand{\div}{\operatorname{div}}

\newcommand{\tr}{\operatorname{tr}}

\newcommand{\bs}{\boldsymbol}

\newcommand{\lap}{\Delta}
\newcommand{\pd}{\partial}

\newtheorem{theorem}{Theorem}[section]
 
\newtheorem{lemma}[theorem]{Lemma}
